\documentclass[10pt]{amsart}
\usepackage{amssymb,amsbsy,amsmath,amsfonts,times, amscd}
\usepackage{latexsym,euscript,exscale}

\usepackage{helvet}

\title{Global and local boundedness of Polish groups}
\author {Christian Rosendal}
\address{Department of Mathematics, Statistics, and Computer Science (M/C 249)\\
University of Illinois at Chicago\\
851 S. Morgan St.\\
Chicago, IL 60607-7045\\
USA}
\email{rosendal.math@gmail.com}
\urladdr{http://homepages.math.uic.edu/$~$rosendal}

\date {}
\newcommand{\norm}[1]{\lVert#1\rVert}
\newcommand{\Norm}[1]{\big\lVert#1\big\rVert}

\newcommand{\triple}[1]{|\!|\!|#1|\!|\!|}

\newcommand{\aff}[1]{{\rm Isom}_{\rm Aff}(#1)}
\newcommand{\lin}[1]{{\rm Isom}_{\rm Lin}(#1)}
\newcommand{\Aff}[1]{{\bf Aff}(#1)}
\newcommand{\GL}[1]{{\bf GL}(#1)}

\newcommand{\acl}[1]{{\bf acl}(#1)}

\newcommand {\B}{\go B}

\newcommand {\F}{\mathbb F}
\newcommand {\N}{\mathbb N}
\newcommand {\M}{\mathbb M}
\newcommand {\Q}{\mathbb Q}
\newcommand {\R}{\mathbb R}
\newcommand {\Z}{\mathbb Z}
\newcommand {\C}{\mathbb C}
\newcommand {\U}{\mathbb U}
\newcommand {\T}{\mathbb T}

\newcommand{\eps}{\epsilon}

\newcommand{\con}{\;\hat{}\;}

\newcommand{\tom} {\emptyset}

\newcommand{\saa}{\Rightarrow}
\newcommand{\equi}{\Longleftrightarrow}

\newcommand{\til}{\rightarrow}

\newcommand{\Lim}[1]{\mathop{\longrightarrow}\limits_{#1}}

\newcommand {\Del}{ \; \Big| \;}
\newcommand {\del}{ \; \big| \;}

\newcommand {\go} {\mathfrak}
\newcommand {\ku} {\mathcal}

\newcommand{\inv}{^{-1}}
\newcommand{\Id}{{\rm Id}}

\newcommand {\e} {\exists}
\renewcommand {\a} {\forall}

\newtheorem{thm}{Theorem}[section]
\newtheorem{cor}[thm]{Corollary}
\newtheorem{lemme}[thm]{Lemma}
\newtheorem{prop} [thm] {Proposition}
\newtheorem{defi} [thm] {Definition}

\newtheorem{claim}[thm] {Claim}

\theoremstyle{definition}

\newtheorem{exa}[thm]{Example}
\newtheorem{prob}[thm]{Problem}

\usepackage{fouriernc}

\begin{document}
\subjclass[2000]{Primary: 22A25, Secondary: 03E15, 46B04}

\keywords{Property (OB), Roelcke precompactness, Strongly non-locally compact groups, Affine  actions on Banach spaces}
\thanks{The research for this article was partially supported by NSF grants DMS 0919700 and DMS 0901405}
\thanks{The author is thankful for many enlightening conversations on the topic of the paper with T. Banakh, J. Melleray, V. Pestov, S. Solecki and T. Tsankov}

\maketitle
\begin{abstract}
We present a comprehensive theory of boundedness properties for Polish groups developed with a main focus on {\em Roelcke precompactness} (precompactness of the lower uniformity) and {\em Property (OB)} (boundedness of all isometric actions on separable metric spaces). In particular, these properties are characterised by the orbit structure of isometric actions on metric spaces and isometric or continuous affine representations on separable Banach spaces.
\end{abstract}

\tableofcontents

\section{Global boundedness properties in Polish groups}
We will be presenting and investigating a number of boundedness properties in {\em Polish} groups, that is, separable, completely metrisable groups, all in some way capturing a different aspect of compactness, but without actually being equivalent with compactness. The main cue for our study comes from the following result, which reformulates compactness of Polish groups.

\begin{thm}\label{compact} The following are equivalent for a Polish group $G$.
\begin{enumerate}
\item $G$ is compact,
\item for any open $V\ni 1$ there is a finite set $F\subseteq G$ such that $G=FV$,
\item for any open $V\ni 1$ there is a finite set $F\subseteq G$ such that $G=FVF$,
\item whenever $G$ acts continuously and by affine isometries on a Banach space $X$, $G$ fixes a point of $X$.
\end{enumerate}
\end{thm}

The implication from (1) to (4) follows, for example, from the Ryll-Nardzewski fixed point theorem and the equivalence of (1) and (2) are probably part of the folklore. On the other hand, the implication from (3) to (1) was shown independently by S. Solecki \cite{actions} and V. Uspenski\u\i{}  \cite{uspenskii}, while the implication from (4) to (1) is due to  L. Nguyen Van Th\'e and V. Pestov \cite{pestov}.

The groups classically studied in representation theory and harmonic analysis are of course the locally compact (second countable), but many other groups of transformations appearing in analysis and elsewhere fail to be locally compact, e.g., homeomorphism groups of compact metric spaces, diffeomorphism groups of manifolds, isometry groups of separable complete metric space, including Banach spaces, and automorphism  groups of countable first order structures. 
While the class of Polish groups is large enough to encompass all of these, it is nevertheless fairly well behaved and allows for rather strong tools, notably Baire category methods, though not in general Haar measure. As it is also reasonably robust, i.e.,  satisfies strong closure properties, it has received considerable attention for the last twenty years, particularly in connection with the descriptive set theory of continuous actions on Polish spaces \cite{becker}. 

The goal of the present paper is to study a variety of global boundedness properties of Polish groups. While this has been done for general topological groups in the context of uniform topological spaces, e.g., by J. Hejcman \cite{hejcman}, one of the most important boundedness properties, namely, Roelcke precompactness  has not received much attention until recently (see, e.g., \cite{uspenskii1, uspenskii2, uspenskii3, uspenskii,  tsankov, glassner}). Moreover, another of these, namely, property (OB) (see \cite{OB}) is not naturally a property of uniform spaces, but nevertheless has a number of equivalent reformulations, which makes it central to our study here.

The general boundedness properties at stake are, on the one hand, precompactness and boundedness of the four natural uniformities on a topological group, namely, the two-sided, left, right and Roelcke uniformities. On the other hand, we have boundedness properties defined in terms of actions on various spaces, e.g., (reflexive) Banach spaces, Hilbert space or  complete metric spaces.

Though we shall postpone the exact definitions till later in the paper, most of these can be simply given as in conditions (2) and (3) of Theorem \ref{compact}. For this and to facilitate the process of keeping track of the different notions of global boundedness of Polish groups, we include a diagram, Figure 1, which indicates how to recover a Polish group $G$ from any open set $V\ni 1$ using a finite set $F\subseteq G$ and a natural number $k$ (both depending on $V$).

\setlength{\unitlength}{1.2cm}
\begin{figure}
\begin{picture}(10,4.4)(-2.5,0.8)
\put(2,0.6){$\bullet$}
\put(2,1.3){$\bullet$}
\put(2,2){$\bullet$}
\put(1,3){$\bullet$}
\put(2,4){$\bullet$}
\put(3,5){$\bullet$}
\put(3.5,4.5){$\bullet$}
\put(4,4){$\bullet$}
\put(3,3){$\bullet$}

\put(2.3,0.6){\footnotesize Property (FH)}
\put(2.3,1.3){\footnotesize Property (ACR)}
\put(2.3,2){\footnotesize Property (OB): $G=(VF)^k$}
\put(-2,3){\footnotesize{Property ($\text{OB}_m$): $G=(VF)^m$}}
\put(-1.3,4){\footnotesize Roelcke precompact: $G=VFV$}
\put(3.3,5){\footnotesize Compact: $G=FVF=FV$}
\put(3.8,4.5){\footnotesize $\ku E_{ts}$-bounded}
\put(4.3,4){\footnotesize Bounded: $G=FV^k$}
\put(3.3,3){\footnotesize Roelcke bounded: $G=V^kFV^k$}

\put(1.07,3.07){\line(1,1){2}}
\put(2.07,2.07){\line(1,1){2}}

\put(2.07,2.07){\line(-1,1){1}}
\put(3.07,3.07){\line(-1,1){1}}
\put(4.07,4.07){\line(-1,1){1}}

\put(2.06,0.67){\line(0,1){1.4}}

\end{picture}
\caption{Implications between various boundedness properties}
\end{figure}
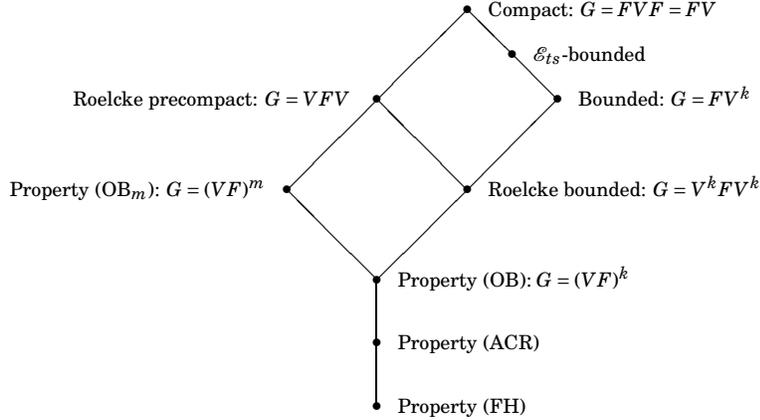

Referring to Figure 1 for the definition of property (OB), our first result characterises this in terms of affine and linear actions on Banach spaces.
\begin{thm}The following conditions are equivalent for a Polish group $G$.
\begin{enumerate}
\item $G$ has property (OB),
\item whenever $G$ acts continuously by affine isometries on a separable Banach space, every orbit is bounded,
\item any continuous linear representation $\pi\colon G\til \GL X$ on a separable Banach space is bounded, i.e., $\sup_{g\in G}\norm{\pi(g)}<\infty$.
\end{enumerate}
\end{thm}
Examples of Polish groups with property (OB) include many transformation groups of highly homogeneous mathematical models, e.g., homeomorphism groups of spheres ${\rm Homeo}(S^n)$ and of the Hilbert cube ${\rm Homeo}([0,1]^\N)$  \cite{OB}.

While property (FH), that is, the requirement that every continuous affine isometric action on a Hilbert space has a fixed point, delineates a non-trivial subclass of locally compact groups, by a result of U. Haagerup and A. Przybyszewska \cite{haagerup}, if one instead considers affine isometric actions on reflexive spaces, one obtains instead simply the class of compact groups. We shall show that the same holds when, rather than changing the space on which the group acts, one considers just continuous affine actions.  

\begin{thm}
Any non-compact locally compact Polish group acts continuously by affine transformations on a separable Hilbert space such that all orbits are unbounded.
\end{thm}

As a corollary, we also obtain information for more general Polish groups.

\begin{cor}Suppose $G$ is a Polish group and $W\leqslant G$ is an open subgroup of 
infinite index with $G= {\rm Comm}_G (W )$. Then $G$ admits a continuous affine representation 
on a separable Hilbert space for which every orbit is unbounded. 
\end{cor}

A second part of our study deals with local versions of these boundedness properties. More exactly, Solecki \cite{haar null} asked whether the class of locally compact groups could be characterised among the Polish groups as those for which there is a neighbourhood of the identity $U\ni 1$ such that  any other neighbourhood $V\ni1$ covers $U$ by a finite number of two-sided translates (he also included a certain additional technical condition of having a free subgroup at $1$ that we shall come back to).  While positive results were obtained by M. Malicki in \cite{malicki}, we shall show that this is not so by presenting a non-locally compact Polish group with a free subgroup at $1$ satisfying the above mentioned covering property for some neighbourhood $U\ni 1$.

\begin{thm}
There is a non-locally compact, Weil complete Polish group, having a free subgroup at $1$ and an open subgroup $U$ whose conjugates $fUf\inv$ provide a neighbourhood basis at $1$. 
\end{thm}

The third and final part of our study deals with the consequences at the microscopic level of the previously mentioned  global boundedness properties. By this we understand not only what happens in a single neighbourhood of the identity, but rather what happens as one decreases the neighbourhood to $1$. As it turns out, the stronger global boundedness properties, namely, Roelcke precompactness (cf. Figure 1) and {\em oligomorphic}, which is a specific form of Roelcke precompactness, prevent further covering properties at the microscopic level. Moreover, this can in turn be utilised in the construction of affine isometric action with non-trivial local dynamics.

We recall that $S_\infty$ is the Polish group consisting of all permutations of the infinite discrete set $\N$ equipped with the topology of pointwise convergence.
\begin{thm}Let $G$ be an oligomorphic closed subgroup of $S_\infty$. Then there is a neighbourhood basis at $1$, $V_0\supseteq V_1\supseteq V_2\supseteq \ldots\ni 1$ such that 
$$
G\neq\bigcup_nF_nV_nE_n
$$
for all finite subsets $F_n,E_n\subseteq G$.

It follows that $G$ admits a continuous affine isometric action on a separable Banach space $X$ such that for some $\eps_n>0$ and all sequences of compact subsets $C_n\subseteq X$ there is a $g\in G$ satisfying
$$
{\rm dist}(gC_n,C_n)>\eps_n
$$
for all $n\in \N$.
\end{thm}

Here a closed subgroup $G\leqslant S_\infty$ is said to be {\em oligomorphic} if, for every $n\geqslant 1$, $G$ induces only finitely many distinct orbits on $\N^n$. By a classical theorem of model theory, up to isomorphism these are exactly the automorphism groups of countable $\aleph_0$-categorical structures, e.g., $S_\infty$, ${\rm Aut}(\Q,<)$ and the homeomorhism group of Cantor space ${\rm Homeo}(2^\N)$ among many other.

\begin{thm}
Suppose $G$ is a non-compact, Roelcke precompact Polish group. Then there is a neighbourhood basis at $1$, $V_0\supseteq V_1\supseteq \ldots \ni 1$ such that for any $h_n\in G$ and finite $F_n\subseteq G$,
$$
G\neq \bigcup_n F_nV_nh_n.
$$
\end{thm}

Again, the class of Roelcke precompact Polish groups is surprisingly large despite of being even more restrictive those with property (OB). As shown in Proposition \ref{roelcke} extending work of \cite{OB, tsankov}, the Roelcke precompact Polish groups are exactly those that can realised as approximately oligomorhic groups of isometries. This criterion immediately gives us the following range of examples, ${\rm Aut}([0,1],\lambda)$ (see \cite{glassner} for an independent proof), $\ku U(\ell_2)$, and less obviously, ${\rm Isom}(\U)$ \cite{OB}, where $\U$ denotes the so called Urysohn metric space.

We should also mention that though we mainly consider Polish groups, many of our results are valid with only trivial modifications for arbitrary topological groups. However, to avoid complications and to get the cleanest statements possible, we have opted for this more restrictive setting, which nevertheless already includes most of the groups appearing in analysis.

The paper is organised as follows: In Sections \ref{uniformities}, \ref{affine} and \ref{topologies} we present some background material on uniformities on topological groups and general constructions of affine and linear representations on Banach spaces. Almost all of the material there is well-known, but sets the stage for several of the constructions used later on.   Sections 1.4--1.10 contains the core study of the various boundedness properties and their consequences. In Sections 2.1 and 2.2, we answer Solecki's question on the possible characterisation of locally compact Polish groups. And finally, in Sections 3.1--3.5 we study the covering properties of neighbourhood bases in Polish groups, which leads to constructions of affine isometric actions with interesting local dynamics.


\subsection{Uniformities and compatible metrics}\label{uniformities}
Recall that a {\em uniform space} is a tuple $(X,\ku E)$, where $X$ is a set and $\ku E$ is a collection of subsets of $X\times X$, called {\em entourages} of the diagonal $\Delta=\{(x,x)\in X\times X\del x\in X\}$, satisfying
\begin{enumerate}
\item $\Delta\subseteq V$ for any $V\in \ku E$,
\item $\ku E$ is closed under supersets, i.e., $V\subseteq U$ and $V\in \ku E$ implies $U\in \ku E$,
\item $V\in \ku E$ implies that also $V\inv=\{(y,x)\in X\times X\del (x,y)\in V\}\in \ku E$,
\item $\ku E$ is closed under finite intersections, i.e., $V,U\in \ku E$ implies that $V\cap U\in \ku E$,
\item for any $V\in \ku E$ there is $U\in \ku E$ such that 
$$
U^2=U\circ U=\{(x,y)\in X\times X\del \e z\in X\; (x,z)\in U\; \&\; (z,y)\in U\}\subseteq V.
$$
\end{enumerate}

The basic example of a uniform space is the case when $(X,d)$ is a metric space (or just pseudometric)  and we let $\ku B$ on $X$ denote the family of  sets 
$$
V_\eps=\{(x,y)\in X\times X\del d(x,y)<\eps\},
$$
for $\eps>0$. Closing $\ku B$ under supersets, one obtains a uniformity $\ku E$ on $X$, and we say that $\ku B=\{V_\eps\}_{\eps>0}$ forms a {\em fundamental system} for $\ku E$, meaning that any entourage contains a subset belonging to $\ku B$.

Conversely, if $(X,\ku E)$ is a uniform space, then $\ku E$ generates a unique topology on $X$ by declaring the vertical sections of entourages at $x$, i.e., $V[x]=\{y\in X\del (x,y)\in V\}$, to form a neighbourhood basis at $x\in X$.

A net $(x_i)$ in $X$ is said to be {\em $\ku E$-Cauchy} provided that for any $V\in \ku E$ we have $(x_i,x_j)\in V$ for all sufficiently large $i,j$. And $(x_i)$ {\em converges} to $x$ if, for any $V\in \ku E$, we have $(x_i,x)\in V$ for all sufficiently large $i$. Thus, $(X,\ku E)$ is {\em complete} if any $\ku E$-Cauchy net converges in $X$.

Similarly, $(X,\ku E)$ is {\em precompact} if for any $V\in \ku E$ there is a finite set $F\subseteq X$ such that 
$$
X=V[F]=\{y\in X\del \e x\in F\;\; (x,y)\in V\}.
$$
That is, $X$ is a union of finitely many  vertical sections $V[x]$ of $V$.

An {\em \'ecart} or {\em pseudometric} on a set $X$ is a symmetric function $d\colon X\times X\til \R_{\geqslant 0}$ satisfying the triangle inequality, $d(x,y)\leqslant d(x,z)+d(z,y)$, and such that $d(x,x)=0$. The Birkhoff-Kakutani theorem states that if $\Delta\subseteq U_n\subseteq X\times X$ is a decreasing sequence of symmetric sets  satisfying 
$$
U_{n+1}\circ U_{n+1}\circ U_{n+1}\subseteq U_n
$$
and we define $\delta, d\colon X\times X\til \R_{\geqslant 0}$ by
$$
\delta(x,y)=\inf \;\{2^{-n}\del (x,y)\in U_n\}
$$
and 
$$
d(x,y)=\inf\big \{\sum_{i=1}^n\delta(x_{i-1},x_i)\del x_0=x, \;x_n=y\big\},
$$
then $d$ is an \'ecart on $X$ with 
$$
\frac 12\delta(x,y)\leqslant  d(x,y)\leqslant \delta(x,y).
$$
In other words, if the $V_\eps$ are defined as above, then $V_{2^{-(n+1)}}\subseteq U_{n}\subseteq V_{2^{-n}}$
and thus the two families $\{U_n\}_{n\in \N}$ and $\{V_\eps\}_{\eps>0}$ are fundamental systems for the same uniformity on $X$. In particular, this shows that any uniformity with a countable fundamental system can be induced by an \'ecart on $X$.

Now, if $G$ is a topological group, it naturally comes with four uniformities, namely, the {\em two-sided},  {\em left},  {\em right} and {\em Roelcke} uniformities denoted respectively $\ku E_{ts}, \ku E_l,\ku E_r$ and $\ku E_R$. These are the uniformities with fundamental systems given by respectively
\begin{enumerate}
\item $E^{ts}_W=\{(x,y)\in G\times G\del x\inv y\in W\;\&\; xy\inv \in W\}$,
\item  $E^{l}_W=\{(x,y)\in G\times G\del x\inv y\in W\}$,
\item  $E^{r}_W=\{(x,y)\in G\times G\del xy\inv\in W\}$,
\item  $E^{R}_W=\{(x,y)\in G\times G\del y\in WxW\}$,
\end{enumerate}
where $W$ varies over symmetric neighbourhoods of $1$ in $G$. Since clearly, $E^{ts}_W\subseteq E^{l}_W\subseteq E^{R}_W$ and $E^{ts}_W\subseteq E^{r}_W\subseteq E^{R}_W$, we see that $\ku E_{ts}$ is finer that both $\ku E_{l}$ and $\ku E_{r}$, while $\ku E_{R}$ is coarser than all of them. In fact, in the lattice of uniformities on $G$, one has $\ku E_{ts}=\ku E_{l}\vee \ku E_{r}$ and $\ku E_{R}=\ku E_{l}\wedge \ku E_{r}$.

Though these uniformities are in general distinct, they all generate the original topology on $G$. This can be seen by noting first that for any symmetric open $W\ni1$ and $x\in G$, one has $\ku E_W^{ts}[x]=xW\cap Wx$, which is a neighbourhood of $x$ in $G$, and thus the topology generated by $\ku E_{ts}$ is coarser than the topology on $G$. Secondly, if $U$ is an open neighbourhood of $x$ in $G$, then there is a symmetric open $W\ni 1$ such that $E^R_W[x]=WxW\subseteq U$, whence the Roelcke uniformity generates a topology as fine as the  topology on $G$.

Note also that if $G$ is first countable, then each of the above uniformities have countable fundamental systems and thus are induced by \'ecarts $d_{ts}, d_l, d_r$ and $d_R$. It follows that each of these induce the topology on $G$ and hence in fact must be metrics on $G$. Moreover, since the sets $E^l_W$ are invariant under multiplication on the left, the uniformity $\ku E_l$ has a countable fundamental system of left-invariant sets, implying that the metric $d_l$ can be made left-invariant. Similarly, $d_r$ can be made right-invariant and, in fact, one can set $d_r(g,h)=d_l(g\inv, h\inv)$. Moreover, since $\ku E_{ts}=\ku E_l\vee \ku E_r$, one sees that $d_l+d_r$ is a compatible metric for the uniformity $\ku E_{ts}$ and thus one can choose $d_{ts}=d_l+d_r$.

A topological group is said to be {\em Raikov} complete if it is complete with respect to the two-sided uniformity and {\em Weil} complete if complete with respect to the left uniformity. This is equivalent to the completeness of the metrics $d_{ts}$ and $d_l$ respectively. Polish groups are always Raikov complete. On the other hand, Weil complete Polish groups are by the above exactly those that admit a compatible, complete, left-invariant metric, something that fails in general.

A function $\varphi\colon X\til \R$ defined on a uniform space $(X,\ku E)$ is {\em uniformly continuous} if for any $\eps>0$ there is some $V\in \ku E$ such that
$$
(x,y)\in V\saa |\varphi(x)-\varphi(y)|<\eps.
$$

The following lemma is well-known (see, e.g., Theorem 1.14 \cite{hejcman} and Theorem 2.4 \cite{atkin}), but we include the simple proof for completeness. 
\begin{lemme}\label{bounded uniformity}
Let $(X,\ku E)$ be a uniform space. Then any uniformly continuous function $\varphi\colon X\til \R$ is bounded if and only if for any $V\in \ku E$ there is a finite set $F\subseteq X$ and an $n$  such that $X=V^n[F]$.
\end{lemme}

\begin{proof}
Suppose first that for any $V\in \ku E$ there is a finite set such that $X=V^n[F]$ and that  $\varphi\colon X\til \R$ is uniformly continuous. Fix $V\in \ku E$ such that $|\varphi(x)-\varphi(y)|<1$ whenever $(x,y)\in V$ and pick a corresponding finite set $F\subseteq X$. Then for any $z\in X$ there are $y_0,\ldots, y_n$ such that $(y_i,y_{i+1})\in V$ and $y_0\in F$, $y_n=z$, whence
\[\begin{split}
|\varphi(x)-\varphi(z)|&=|\varphi(y_0)-\varphi(y_n)|\\
&\leqslant |\varphi(y_0)-\varphi(y_1)|+|\varphi(y_1)-\varphi(y_2)|+\ldots+|\varphi(y_{n-1})-\varphi(y_n)|\\
&<  n.
\end{split}\]
Since $F$ is finite, it follows that $\varphi$ is bounded.

Suppose conversely that $V\in \ku E$ is a symmetric set such that for all $n$ and finite $F\subseteq X$, $X\neq V^n[F]$. Assume first that there is some $x\in X$ such that $V^n[x]\subsetneq V^{n+1}[x]$ for all $n\geqslant 1$ and extend $(V^{3n})_{n\geqslant 1}$ to a bi-infinite sequence $(U_n)_{n\in \Z}$ of symmetric sets in $\ku E$ such that $U_n^3\subseteq U_{n+1}$ for all $n\in \Z$. Defining $\delta,d\colon X\times X\til \R_{\geqslant 0}$ by 
$$
\delta(x,y)=\inf \;\{2^{n}\del (x,y)\in U_n\}
$$
and 
$$
d(x,y)=\inf\big \{\sum_{i=1}^n\delta(x_{i-1},x_i)\del x_0=x, \;x_n=y\big\},
$$
as in the Birkhoff-Kakutani theorem, we get that $\frac 12\delta\leqslant  d\leqslant\delta$. Moreover, $\varphi(y)=d(x,y)$ defines a uniformly continuous function on $X$. To see that $\varphi$ is unbounded, for any $n$ it suffices to pick some $y\notin V^{3n}[x]=U_n[x]$, i.e., $(x,y)\notin U_n$ and thus $\varphi(y)\geqslant \frac 12\delta(x,y)>2^n$. 

Suppose, on the other hand, that for any $x\in X$ there is some $n_x$ such that $V^{n_x}[x]=V^{n_x+1}[x]$. Then, by the symmetry of $V$, for any $x,y$ either $V^{n_x}[x]=V^{n_y}[y]$ or $V^{n_x}[x]\cap V^{n_y}[y]=\tom$. Picking inductively $x_1,x_2,\ldots$ such that $x_{k+1}\notin V^{n_{x_1}}[x_1]\cup\ldots\cup V^{n_{x_k}}[x_k]$, the $V^{n_{x_k}}[x_k]$ are all disjoint and we can therefore let $\varphi$ be constantly equal to $k$ on $V^{n_{x_k}}[x_k]$ and $0$ on $X\setminus \bigcup_{k\geqslant 1}V^{n_{x_k}}[x_k]$. Then $\varphi$ is unbounded but uniformly continuous.
\end{proof}


\subsection{Constructions of linear and affine actions on Banach spaces}\label{affine}
Fix a non-empty set $X$ and let $c_{00}(X)$ denote the vector space of finitely supported functions $\xi\colon X\til \R$. The subspace $\M(X)\subseteq c_{00}(X)$ consists of all $m\in c_{00}(X)$ for which 
$$
\sum_{x\in X}m(x)=0.
$$
Alternatively, $\M(X)$ is the hyperplane in $c_{00}(X)$ given as the kernel of the functional $m\mapsto \sum_{x\in X}m(x)$.
The elements of $\M(X)$ are called {\em molecules} and basic among these are the {\em atoms}, i.e., the molecules of the form
$$
m_{x,y}=\delta_x-\delta_y,
$$
where $x,y\in X$ and $\delta_x$ is the Dirac measure at $x$. As can easily be seen by induction on the cardinality of its support, any molecule $m$ can be written as a finite linear combination of atoms, i.e.,  
$$
m=\sum_{i=1}^na_im_{x_i,y_i},
$$
for some $x_i,y_i\in X$ and $a_i\in \R$.

Also, if $G$ is a group acting on $X$, one obtains an action of $G$ on $\M(X)$ by linear automorphisms, i.e., a linear representation $\pi\colon G\til \GL{\M(X)}$, by setting 
$$
\pi(g) m=m(g\inv\;\cdot\;),
$$
whence
$$
\pi(g)\big( \sum_{i=1}^na_im_{x_i,y_i}\big)=\sum_{i=1}^na_im_{gx_i,gy_i}
$$ 
for any molecule $m=\sum_{i=1}^na_im_{x_i,y_i}\in \M(X)$ and $g\in G$.

Suppose now $Z$ is an $\R$-vector space and consider the group $\Aff Z$ of affine automorphisms of $Z$. This splits as a semidirect product
$$
\Aff Z=\GL Z\ltimes Z,
$$
that is, $\Aff Z$ is isomorphic to the Cartesian product $\GL Z \times Z$ with the group multiplication 
$$
(T,x)*(S,y)=(TS,x+Ty).
$$
Equivalently, the action of $(T,x)$ on $Z$ is given by $(T,x)(z)=Tz+x$.
Therefore, if $\rho\colon G\til \Aff Z$ is a homomorphism from a group $G$, it decomposes as $\rho=\pi\times b$, where $\pi\colon G\til  \GL Z$ is a homomorphism and $b\colon G\til Z$ satisfies the {\em cocycle relation}
$$
b(gh)=b(g)+\pi(g)\big(b(h)\big).
$$
In this case, we say that $b$ is a {\em cocycle} associated to $\pi$ and note that the affine action of $G$ on $Z$ corresponding to $\rho$ has a fixed point on $Z$ if and only if $b$ is a {\em coboundary}, i.e., if there is some $x\in Z$ such that $b(g)= \pi(g)x-x$.

Returning to our space of molecules, suppose $G$ acts on the set $X$ and let $\pi\colon G\til \GL{\M(X)}$  denote the linear representation of $G$ given by $\pi(g)m=m(g\inv\;\cdot\;)$. Now, for any point $e\in X$, we let $\phi_e\colon X\til \M(X)$ be the injection defined by 
$$
\phi_e(x)=m_{x,e}
$$
and construct a cocycle $b_e\colon G\til \M(X)$ associated to $\pi$ by setting 
$$
b_e(g)=m_{ge,e}.
$$
To verify the cocycle relation, note that for $g,h\in G$
$$
b_e(gh)=m_{ghe,e}=m_{ge,e}+m_{ghe,ge}=b_e(g)+\pi(g)\big(b_e(h)\big).
$$
We let $\rho_e\colon G\til \Aff {\M(X)}$ denote the corresponding affine representation $\rho_e=\pi\times b_e$ of $G$ on $\M(X)$.

With these choices, it is easy to check that for any $g\in G$ the following diagram commutes.
$$
\begin{CD}
X @>g>> X\\
@V\phi_eVV   @VV\phi_eV\\
\M(X)@>>\rho_e(g)> \M(X)
\end{CD}
$$
Indeed, for any $x\in X$,
$$
\big(\rho_e(g)\circ \phi_e\big)(x)=\rho_e(g)(m_{x,e})=\pi(g)(m_{x,e})+b_e(g)=m_{gx,ge}+m_{ge,e}=m_{gx,e}
=\big(\phi_e\circ g\big)(x).
$$


Now, if $\psi$ is a {\em non-negative kernel} on $X$, that is, a function $\psi\colon X\times X\til \R_{\geqslant 0}$, one can define a pseudonorm on $\M(X)$, by the formula
$$
{\norm {m}}_\psi=\inf\big( \sum_{i=1}^n |a_i| \psi(x_i,y_i) \del m=\sum_{i=1}^na_im_{x_i,y_i}\big).
$$

Thus, if $\psi$ is $G$-invariant, one sees that $\pi\colon G\til \GL{\M(X)}$ corresponds to an action by linear isometries on $(\M(X),\norm\cdot_\psi)$ and so the action extends to an isometric action on the completion of $\M(X)$ with respect to $\norm\cdot_\psi$. On the other hand, if $\psi$ is no longer $G$-invariant, but instead satisfies
$$
\psi(gx,gy)\leqslant K_g\psi(x,y)
$$
for  all $x,y\in X$ and some constant $K_g$ depending only on $g\in G$, then every operator $\pi(g)$ is bounded, $\norm{\pi(g)}_\psi\leqslant K_g$, and so again the action of $G$ extends to an action by bounded automorphisms on the completion of $(\M(X),\norm\cdot_\psi)$.

A special case of this construction is when $\psi$ is a metric $d$ on $X$, in which case we denote the resulting {\em Arens-Eells} norm by $\norm\cdot_{\AE}$ instead of $\norm\cdot_\psi$. An easy exercise using the triangle inequality shows that in the computation of the norm by 
$$
{\norm {m}}_{\AE}=\inf\big( \sum_{i=1}^n |a_i| d(x_i,y_i) \del m=\sum_{i=1}^na_im_{x_i,y_i}\big),
$$
the infimum is attained at some presentation $m=\sum_{i=1}^na_im_{x_i,y_i}$ where $x_i$ and $y_i$ all belong to the support of $m$.
Moreover, as is well-known (see, e.g., \cite{weaver}), the norm is equivalently computed by
$$
\norm m_{\AE}=\sup\big( \sum_{x\in X}m(x)f(x)\del f\colon X\til \R \text{ is $1$-Lipschitz }\big),
$$
and so, in particular, $\norm{m_{x,y}}_{\AE}=d(x,y)$ for any $x,y\in X$.

We denote the completion of $\M(X)$ with respect to $\norm\cdot_{\AE}$ by $\AE(X)$, which we call the {\em Arens-Eells} space of $(X,d)$. 
It is not difficult to verify that the set of molecules that are rational linear combinations of atoms with support in a dense subset of $X$ is dense in $\AE(X)$ and thus, provided $X$ is separable,  $\AE(X)$ is a separable Banach space.   A fuller account of the Arens-Eells space can be found in the book by N. Weaver \cite{weaver}.

Now, recall that by the Mazur--Ulam Theorem, any surjective isometry between two Banach spaces is affine and hence, in particular, the group of all isometries of a Banach space $Z$ coincides with the group of all affine isometries of $Z$. To avoid confusion, we shall denote the latter by $\aff Z$ and let $\lin Z$ be the subgroup consisting of all linear isometries of $Z$. By the preceding discussion, we have $\aff Z=\lin Z\ltimes Z$.


\subsection{Topologies on transformation groups}\label{topologies}
Recall that if $X$ is a Banach space and ${\bf B}(X)$ the algebra of bounded linear operators on $X$, the {\em strong operator topology} (SOT) on ${\bf B}(X)$ is just the topology of pointwise convergence on $X$, that is, if $T_i, T\in {\bf B}(X)$, we have  
$$
T_i\Lim{\scriptsize{\rm SOT}} T\;\equi\; \norm{T_ix- Tx}\til 0 \text{ for all }x\in X.
$$
In general, the operation of composition of operators is not strongly (i.e., SOT) continuous, but if one restricts it to a norm bounded subset of ${\bf B}(X)$ it will be. 

Similarly, if we restrict to a norm bounded subset of $\GL X \subseteq {\bf B}(X)$, then inversion $T\mapsto T\inv$ is also strongly continuous and so, in particular, $\lin X\subseteq \GL X$ is a topological group with respect to the strong operator topology. 
In fact, provided $X$ is separable, $\lin X$ is a Polish group in the strong operator topology.

Recall that an action $G\curvearrowright X$ of a topological group $G$ on a topological space $X$ is {\em continuous} if it is jointly continuous as a map from $G\times X$ to $X$. Now, as can easily be checked by hand, if $G$ acts by isometries on a metric space $X$, then joint continuity of $G\times X\til X$ is equivalent to the map $g\in G\mapsto gx\in X$ being continuous for every $x\in X$.

Thus, an action of a topological group $G$ by linear isometries on a Banach space $X$ is continuous if and only if the corresponding representation $\pi\colon G\til \lin X$ is strongly continuous, i.e., if it is continuous with respect to the strong operator topology on $\lin X$.

Since $\GL X$ is not in general a topological group in the strong operator topology, one has to be a bit more careful when dealing with not necessarily isometric representations. 

Assume  first that $\pi\colon G\til \GL X$ is a representation of a Polish group by bounded automorphisms of a Banach space $X$ such that the corresponding action $G\curvearrowright X$ is continuous. We claim that $\norm{\pi(g)}$ is bounded 
in a neighbourhood $U$ of the identity $1$ in $G$. For if not, we could find $g_n\til 1$ such that $\norm{\pi(g_n)}>n^2$ and so for some $x_n\in X$, $\norm{x_n}=1$, we have  $\norm{\pi(g_n)x_n}>n^2$. But then $z_n=\frac {x_n}n\til 0$, while $\norm{\pi(g_n)z_n}>n$, contracting that $\pi(g_n)z_n\til \pi(1)0=0$ by continuity of the action.  Moreover, $\pi$ is easily seen to be strongly continuous.

Conversely, assume that $\pi\colon G\til \GL X$ is a strongly continuous representation such that $\norm{\pi(g)}$ is bounded by a constant $K$ in some neighbourhood $U\ni 1$. Then, by strong continuity of $\pi$, if $\eps>0$, $x\in X$ and $g\in G$ are given, we can find a neighbourhood $V\ni 1$ such that $\norm{\pi(vg)x-\pi(g)x}<\eps/2$ for $v\in V$. It follows that if $\norm{y-x}<\frac{\eps}{2K\norm{\pi(g)}}$ and $v\in V\cap U$, then
\[\begin{split}
\norm{\pi(vg)y-\pi(g)x}&\leqslant \norm{\pi(vg)y-\pi(vg)x}+\norm{\pi(vg)x-\pi(g)x}\\
&\leqslant \norm{\pi(v)}\norm{\pi(g)}\norm{y-x}+\eps/2\\
&\leqslant \eps,
\end{split}\]
showing that the action is continuous.

Therefore, a representation $\pi\colon G\til \GL X$ corresponds to a continuous action $G\curvearrowright X$ if and only if (i) $\pi$ is strongly continuous and (ii) $\norm{\pi(g)}$ is bounded in a neighbourhood of $1\in G$. For simplicity, we shall simply designate this by {\em continuity of the representation} $\pi$.

Similarly, a representation $\rho\colon G\til \Aff X$ by continuous affine transformations of $X$ corresponds to a continuous action of $G$ on $X$ if and only if both the corresponding linear representation $\pi\colon G\til \GL X$ and the cocycle $b\colon G\til X$ are continuous.


\subsection{Property (OB)}\label{property (OB)}
Our first boundedness property is among the weakest of those studied and originated in \cite{OB} as a topological analogue of a purely algebraic property initially investigated by G. M. Bergman \cite{bergman}.

\begin{defi}
A topological group $G$ is said to have {\em property (OB)} if whenever $G$ acts continuously by isometries on a metric space, every orbit is bounded.
\end{defi}
Since continuity of an isometric action of $G$ is equivalent to continuity of the maps $g\mapsto gx$ for all $x\in X$,  property (OB) for $G$ can be reformulated as follows: Whenever $G$ acts by isometries on a metric space $(X,d)$, such that for every $x\in X$ the map $g\mapsto gx$ is continuous, every orbit is bounded.

We recall some of the equivalent characterisations of property (OB) for Polish groups, a few of which were shown in  \cite{OB}.

\begin{thm}\label{OB}
Let $G$ be a Polish group. Then the following conditions are equivalent.
\begin{enumerate}
\item $G$ has property (OB),
\item whenever $G$ acts continuously by affine isometries on a separable Banach space, every orbit is bounded,
\item any continuous linear representation $\pi\colon G\til \GL X$ on a separable Banach space is bounded, i.e., $\sup_{g\in G}\norm{\pi(g)}<\infty$,

\item whenever $W_0\subseteq W_1\subseteq W_2\subseteq\ldots \subseteq G$ is an exhaustive sequence of open subsets, then $G=W_n^k$ for some $n,k\geqslant 1$,
\item for any open symmetric $V\neq \tom$ there is a finite set $F\subseteq G$ and some $k\geqslant 1$ such that $G=(FV)^k$,
\item\begin{enumerate}
\item[(i)] $G$ is not the union of a chain of proper open subgroups, and
\item[(ii)] if $V$ is a symmetric open generating set for $G$, then $G=V^k$ for some $k\geqslant 1$.
\end{enumerate}

\item any compatible left-invariant metric on $G$ is bounded,
\item any continuous left-invariant \'ecart on $G$ is bounded,
\item any continuous {\em length} function $\ell\colon G\til \R_+$,  i.e., satisfying $\ell(1)=0$ and $\ell(xy)\leqslant \ell(x)+\ell(y)$, is bounded.
\end{enumerate}
\end{thm}

\begin{proof}All but items (2), (3) , (8) and (9) were shown to be equivalent to (OB) in \cite{OB}. Now,  (8) and (9) easily follow from (4), while (8) implies (7), and left-invariant \'ecarts $d$ and length-functions $\ell$ are dual by $\ell(g)=d(g,1)$ and $d(g,h)=\ell(h\inv g)$.
Also, (2) is immediate from (1). And if $\pi\colon G\til \GL X $ is a continuous linear representation, then $\norm{\pi(g)}$ is bounded in a neighbourhood of $1$. So, if (5) holds, then $\pi$ is bounded, showing (5)$\saa$(3).

(2)$\saa$(1): We show the contrapositive, that is, if  $G$ acts continuously and isometrically on a separable metric space $(X,d)$ with unbounded orbits, then $G$ acts continuously and by affine isometries on the Arens-Eells space $\AE(X)$ space such that every orbit is unbounded. 

As in Section \ref{affine}, let $\pi$ denote the isometric linear representation of $G$ on $\AE(X)$ induced by the action $\pi(g)m=m(g\inv\;\cdot\;)$ on $\M(X)$ and, for any point $e\in X$, construct a cocycle $b_e\colon G\til \AE(X)$ associated to $\pi$ by setting 
$$
b_e(g)=m_{ge,e}.
$$
Define also an isometric embedding $\phi_e\colon X\til \AE(X)$ by 
$$
\phi_e(x)=m_{x,e}.
$$
To verify that $\phi_e$ indeed is an isometry, note that 
$$
\norm{\phi_e(x)-\phi_e(y)}=\norm{(\delta_x-\delta_e)-(\delta_y-\delta_e)}=\norm{\delta_x-\delta_y}=d(x,y).
$$
As noted in Section \ref{affine}, $\phi_e$ conjugates the $G$-action on $X$ with $\rho_e$ and thus, as $G$ has an unbounded orbit on $X$, it follows that $G$ has an unbounded orbit on $\AE(X)$ via the affine isometric action $\rho_e$.


(3)$\saa$(7): Again we show that contrapositive. So suppose $d$ is an unbounded compatible left-invariant metric on $G$ and let  $\sigma\colon G\til [1,\infty\,[$ be the function defined by 
$$
\sigma(g)=\exp d(g,1_G)
$$
and note that $\sigma(1_G)=1$, $\sigma(g\inv)=\sigma(g)$ and  $\sigma(gh)\leqslant \sigma(g)\sigma(h)$.
Also, for $g\in G$, let $\pi(g)\in \GL{\M(G)}$ be the invertible linear operator defined by $\pi(g)m=m(g\inv\; \cdot\;)$.

Let now $\psi\colon G^2\til \R_{\geqslant 0}$ be the non-negative kernel on $G$ defined by 
$$
\psi(g,h)=\sigma(g)\sigma(h)d(g,h)
$$ 
and consider the corresponding pseudonorm 
$$
{\norm {m}}_\psi=\inf\big( \sum_{i=1}^n |a_i|\psi(p_i,q_i) \;\del\; m=\sum_{i=1}^na_im_{p_i,q_i}\big).
$$
Note that for any $g,h,f\in G$, one has
$$
\psi(gh,gf)=\sigma(gh)\sigma(gf)d(gh,gf)\leqslant \sigma(g)^2\sigma(h)\sigma(f)d(h,f)=\sigma(g)^2\psi(h,f),
$$
and so $\norm{\pi(g)}_\psi\leqslant \sigma(g)^2$.

\begin{claim}
Suppose $m\in \M(G)$ is a molecule,  $g\in G$ and $\alpha>0$. Then, if $m(h)=0$ for all $h\neq g$ with $d(g,h)<\alpha$, we have
$$
{\norm {m}}_\psi\geqslant |m(g)|\alpha \exp\big( d(g,1)-\alpha\big).
$$
\end{claim}

To see this, let $m=\sum_{i=1}^na_im_{p_i,q_i}$ be any presentation of $m$ and let $A$ be the set of $i\in [1,n]$ such that either $d(g,p_i)< \alpha$ or $d(g,q_i)< \alpha$. Set also $m_1=\sum_{i\in A}a_im_{p_i,q_i}$ and $m_2=\sum_{i\notin A}a_im_{p_i,q_i}$, whence $m_2(h)=0$ whenever $d(g,h)< \alpha$. Since $m=m_1+m_2$, it follows that  $m_1(h)=m(h)=0$ for any $h\neq g$ with $d(g,h)<\alpha$. Moreover, by the definition of $\sigma$, we see that $\sigma(h)>\exp\big( d(g,1)-\alpha\big)$ for any $h\in G$ with $d(g,h)<\alpha$ and so for any $i\in A$,
$$
\exp\big( d(g,1)-\alpha\big)<\sigma(p_i)\sigma(q_i).
$$

Now, by the calculus for the Arens-Eells space, there is a presentation $m_1=\sum_{i=1}^kb_im_{r_i,s_i}$, with $r_i,s_i\in {\rm supp}(m_1)$ and $r_i\neq s_i$, minimising the estimate for the Arens-Eells norm of $m_1$, in particular, such that
$$
\sum_{i=1}^k|b_i|d({r_i,s_i})\leqslant \sum_{i\in A}|a_i|d({p_i,q_i}).
$$
Letting $C$ be the set of $i\in [1,k]$ such that either $r_i=g$ or $s_i=g$, we see that 
$$
|m(g)|=|m_1(g)|=\big|\sum_{i\in C}b_im_{r_i,s_i}(g)\big|\leqslant \sum_{i\in C}|b_i|.
$$
Moreover, for $i\in C$, $d(r_i,s_i)\geqslant \alpha$ and thus 
$$
|m(g)|\alpha\leqslant \sum_{i\in C}|b_i|d(r_i,s_i)\leqslant \sum_{i\in A}|a_i|d({p_i,q_i}).
$$
It thus follows that
$$
 |m(g)| \alpha\exp\big( d(g,1)-\alpha\big)< \sum_{i\in A}|a_i|\sigma(p_i)\sigma(q_i)d({p_i,q_i})\leqslant \sum_{i=1}^n|a_i|\sigma(p_i)\sigma(q_i)d({p_i,q_i}).
$$
Since the presentation $m=\sum_{i=1}^na_im_{p_i,q_i}$ was arbitrary, this shows that
$$
|m(g)| \alpha\exp\big( d(g,1)-\alpha\big)\leqslant \norm m_\psi,
$$
which proves the claim.

Note that then if $m$ is any non-zero molecule, we can choose $g\neq 1$ in its support and let $0<\alpha<\frac 12 d(g,1)$ be such that $m(h)=0$ for any $h\neq g$ with $d(g,h)<\alpha$. Then
$\norm m_\psi\geqslant |m(g)| \alpha\exp\big( d(g,1)-\alpha\big)>0$, which shows that $\norm\cdot_\psi$ is a norm on $\M(G)$.

Also, if $f,h\in G$ with $d(f,h)>1$, then $\norm{m_{f,h}}_\psi\geqslant  \exp\big( d(f,1)-1\big)$. Therefore, if we let $g_n\in G$ be such that $d(1,g_n\inv)\Lim{n\til \infty}\infty$ and pick $f,h\in G$ with $d(f,h)>1$, then 
$$
\norm{\pi({g_n})m_{f,h}}_\psi=\norm{m_{g_nf,g_nh}}_\psi\geqslant  \exp\big( d(g_nf,1)-1\big)= \exp\big( d(f,g_n\inv)-1\big)\Lim{n\til \infty}\infty,
$$
showing that also $\norm{\pi({g_n})}_\psi\Lim{n\til \infty}\infty$.

Also, as is easy to verify, if $D$ is a countable dense subset of $G$, the set of molecules that are rational linear combinations of atoms $m_{g,h}$ with $g,h\in D$ is a countable dense subset of $(\M(G),\norm\cdot_\psi)$. So the completion $Z=\overline{\M(G)}^{\norm\cdot_\psi}$ is a separable Banach space and  $\pi\colon G\til \GL{\M(G),\norm\cdot_\psi}$  extends to a continuous action of $G$ by bounded linear automorphisms on $Z$ with $\norm{\pi(g)}_\psi$ unbounded. 
\end{proof}


\subsection{Bounded uniformities}\label{bdd uniformities}
For each of the uniformities considered in Section \ref{uniformities}, one may consider the class of groups for which they are bounded in the sense of every real valued uniformly continuous function being bounded, cf. Lemma \ref{bounded uniformity}. Of course, such considerations are not new and indeed for the left, or equivalently right uniformity, appeared in the work of J. Hejcman \cite{hejcman} (see also \cite{atkin}).
\begin{defi}\cite{hejcman}
A topological group $G$ is {\em bounded} if for any  open $V\ni 1$ there is a finite set $F\subseteq G$ and some $k\geqslant 1$ such that $G=FV^k$.
\end{defi}

Recall that a function $\varphi\colon G\til \R$ is {\em left-uniformly continuous} if for any $\eps>0$ there is an open $V\ni 1$ such that
$$
\a x,y \in G\; (x\inv y\in V\til |\varphi(x)-\varphi(y)|<\eps).
$$

We then have the following reformulation of boundedness.

\begin{prop}\label{bounded}
The following are equivalent for a Polish group $G$.
\begin{enumerate}
\item $G$ is bounded,
\item any left-uniformly continuous $\varphi\colon G\til \R$ is bounded,
\item $G$ has property (OB) and any open subgroup has finite index.
\end{enumerate}
\end{prop}

\begin{proof}(1)$\saa$(2): This implication is already contained in Lemma \ref{bounded uniformity}.

(2)$\saa$(3): If $H\leqslant G$ is an open subgroup with infinite index, let $x_1,x_2,\ldots$ be left coset representatives for $H$ and define $\varphi\colon G\til \R$ by $\varphi(y)=n$ for all $y\in x_nH$. Then $\varphi$ is left-uniformly continuous and unbounded. Also if $G$ fails property (OB) it admits an unbounded continuous length function $\ell\colon G\til \R_+$. But then $\ell$ is also left-uniformly continuous. 

(3)$\saa$(1): Suppose that $G$ has property (OB) and that any open subgroup has finite index. Then whenever $V\ni1$ is open, the group generated $\langle V\rangle=\bigcup_{n\geqslant 1}V^n$ is open and must have finite index in $G$. Now, by Proposition 4.3 in \cite{OB}, also $\langle V\rangle$ has property (OB) and, in particular, there is some $k\geqslant 1$ such that $\langle V\rangle =V^k$. Letting $F\subseteq G$ be a finite set of left coset representatives for $\langle V\rangle$ in $G$, we have $G=FV^k$, showing that $G$ is bounded.
\end{proof}

A function $\varphi\colon G\til \R$ is {\em uniformly continuous} if for any $\eps>0$ there is an open $V\ni 1$ such that
$$
\a x,y \in G\; (y\in VxV\til |\varphi(x)-\varphi(y)|<\eps).
$$
Equivalently, $\varphi$ is uniformly continuous if it simultaneously left and right uniformly continuous, i.e., continuous with respect to the Roelcke uniformity.

\begin{defi}
A topological group $G$ is {\em Roelcke bounded} if for any  open $V\ni 1$ there is a finite set $F\subseteq G$ and some $k\geqslant 1$ such that $G=V^kFV^k$.
\end{defi}

As for boundedness, we have the following reformulations of Roelcke boundedness.

\begin{prop}\label{Roelcke bounded}
The following are equivalent for a Polish group $G$.
\begin{enumerate}
\item $G$ is Roelcke bounded,
\item any uniformly continuous $\varphi\colon G\til \R$ is bounded,
\item 
\begin{enumerate}
\item for any open subgroup $H$, the double coset space ${\rm H}\backslash {\rm G}\,\slash \,{\rm H}$ is finite, and
\item any open subgroup has property (OB).
\end{enumerate}\end{enumerate}
\end{prop}

\begin{proof}
Again the implication from (1) to (2) follows from Lemma \ref{bounded uniformity}.

(2)$\saa$(3): If $H\leqslant G$ is an open subgroup, then any  $\varphi\colon G\til \R$ that is constant on each double coset $HxH$ will be uniformly continuous. So if the double coset space ${\rm H}\backslash {\rm G}\,\slash \,{\rm H}$ is infinite, then $G$ supports an unbounded uniformly continuous function.

Also, if $H\leqslant G$ is an open subgroup without property (OB), then there is an unbounded continuous length function $\ell\colon H\til \R_+$. Setting $\ell(x)=0$ for all $x\in G\setminus H$, $\ell\colon G\til \R$ is uniformly continuous, but unbounded.

(3)$\saa$(1): Assume that (3) holds and that $V\ni 1$ is an open set. Then the open subgroup $\langle V\rangle$ has property (OB) and hence for some $k\geqslant 1$, $\langle V\rangle =V^k$.  Moreover,  the double coset space ${\rm \langle V\rangle}\backslash {\rm G}\,\slash \,{\rm \langle V\rangle}$ is finite and $G=\langle V\rangle F \langle V\rangle=V^kFV^k$ for some finite set $F\subseteq G$.
\end{proof}

Note that if $G$ is a Polish group all of whose open subgroups have finite index, e.g., if $G$ is connected, then the three properties of boundedness, Roelcke boundedness and property (OB) are equivalent. 

By Proposition 4.3 of \cite{OB}, if $G$ is a Polish group with property (OB) and $H\leqslant G$ is an open subgroup of finite index, then $H$ also has property (OB). However, it remains an open problem whether property (OB) actually passes to all open subgroups of (necessarily) countably index. Note that if this where to be the case, condition (3) in Proposition \ref{Roelcke bounded} above would simplify.

\begin{prob}Suppose $G$ is a Polish group with property (OB) and $H\leqslant G$ is an open subgroup. Does $H$ have property (OB)?
\end{prob}

Since a Polish group is easily seen to be bounded for the left uniformity if and only if it is bounded for the right uniformity, the only remaining case is the two-sided uniformity. Unfortunately, other than Lemma \ref{bounded uniformity}, we do not have any informative reformulation of this property for Polish groups. 

\begin{defi}
A topological group $G$ is {\em $\ku E_{ts}$-bounded} if for any  open $V\ni 1$ there is a finite set $F\subseteq G$ and some $k\geqslant 1$ such that for any $g\in G$ there are $x_0,\ldots , x_{k-1}$ and $x_k=g$ such that $x_0\in F$, $x_{i+1}\in x_iV$ and $x_{i+1}\inv \in x_i\inv V$ for all $i$. 
\end{defi}


\subsection{Roelcke precompactness}The notion of Roelcke precompactness originates in the work of W. Roelcke \cite{roelcke} on uniformities on groups and has recently been developed primarily by  Uspenski\u\i{}   \cite{uspenskii1,uspenskii2,uspenskii3,uspenskii} and, in the work of T. Tsankov \cite{tsankov}, fund some very interesting applications in the classification of unitary representations of non-Archimedean Polish groups. Tsankov was also able to essentially characterise $\aleph_0$-categoricity of a countable model theoretical structure in terms of Roelcke precompactness of its automorphism group, thus refining the classical theorem of Engeler, Ryll-Nardzewski and Svenonius (see \cite{hodges}). We present a related characterisation in Proposition \ref{roelcke} below.

\begin{defi}
A topological group $G$ is {\em Roelcke precompact} if and only if for any open $V\ni 1$ there is a finite set $F\subseteq G$ such that $G=VFV$.
\end{defi}

Note that, by Theorem \ref{compact}, precompactness with respect to either of the three other uniformities on a Polish group is simply equivalent to compactness. So Roelcke precompactness is the only interesting notion.

Suppose $G$ is a group acting by isometries on a metric space $(X,d)$. For any $n\geqslant 1$, we let $G$ act diagonally on $X^n$, i.e., 
$$
g\cdot (x_1,\ldots, x_n)=(gx_1,\ldots, gx_n),
$$
and equip $X^n$ with the supremum metric $d_\infty$ defined from $d$ by 
$$
d_\infty\big((x_1,\ldots,x_n),(y_1, \ldots,y_n)\big)=\sup_{1\leqslant i\leqslant n }d(x_i,y_i).
$$

\begin{defi}
An isometric action $\alpha \colon G\curvearrowright X$ by a group $G$ on a metric space $(X,d)$ is said to be {\em approximately oligomorphic} if for any $n\geqslant 1$ and $\eps>0$ there is a finite set $A\subseteq X^n$ such that
$$
G\cdot A=\{g\cdot \overline x\del g\in G\;\;\; \&\;\;\; \overline x\in A\}
$$
is $\eps$-dense in $(X^n,d_\infty)$.
\end{defi}

\begin{prop}\label{roelcke}
The following are equivalent for a Polish group $G$.
\begin{enumerate}
\item $G$ is Roelcke precompact,
\item for every $n\geqslant 1$ and open $V\ni 1$ there is a finite set $F\subseteq G$ such that
$$
\underbrace{G\times \ldots\times G}_{n\textrm { times}}=V\cdot (\underbrace{FV\times \ldots \times FV}_{n \text{ times}})
$$
\item for any continuous isometric action $\alpha\colon G\curvearrowright X$ on a metric space $(X,d)$ inducing a dense orbit, any open $U\ni 1$, $\eps>0$ and $n\geqslant 1$, there is a finite set $A\subseteq X^n$ such that $U\cdot A$ is $\eps$-dense in $(X^n,d_\infty)$,
\item $G$ is topologically isomorphic to a closed subgroup $ H\leqslant {\rm Isom}(X,d)$, where $(X,d)$ is a separable complete metric space, ${\rm Isom}(X,d)$ is equipped with the topology of pointwise convergence and the action of $H$ on $X$ is approximately oligomorphic and induces a dense orbit.
\end{enumerate}
\end{prop}

The implication from (1) to (3) was essentially noted in \cite{tsankov}.

\begin{proof}
(1)$\saa$(2): The proof is by induction on $n\geqslant 1$, the case $n=1$ corresponding directly to Roelcke precompactness.

Now suppose the result hold for $n$ and fix $V\ni 1$ symmetric open. Choose a symmetric open set $W\ni 1$ such that $W^2\subseteq V$ and find by the induction hypothesis some finite set $D\subseteq G$ such that
$$
\underbrace{G\times \ldots\times G}_{n\textrm { times}}=W\cdot (\underbrace{DW\times \ldots \times DW}_{n \text{ times}}).
$$
Set now $U=V\cap \bigcap_{d\in D}dWd\inv$ and pick a finite set $E\subseteq G$ such that $G=UEU$. We claim that for  $F=D\cup E$, we have 
$$
\underbrace{G\times \ldots\times G}_{n+1\textrm { times}}=V\cdot (\underbrace{FV\times \ldots \times FV}_{n+1 \text{ times}}).
$$
To see this, suppose $x_1,\ldots, x_n, y\in G$ are given. By choice of $D$, there are $w\in W$ and $d_1,\ldots, d_n\in D$ such that $x_i\in wd_iW$ for all $i=1,\ldots, n$. Now find some $u\in U$ such that $w\inv y\in uEU$, whence  $y\in wuEV$. Since $u\inv \in U\subseteq dWd\inv$ for every $d\in D$, we have $d_i\inv u\inv d_i\in W$ for every $i$ and so 
$$
x_i\in wd_iW=wd_i\cdot d_i\inv u d_i\cdot d_i\inv u\inv d_i\cdot W\subseteq w ud_i W^2\subseteq w ud_iV.
$$ 
Thus, 
$$
(x_1,\ldots, x_n,y)\in wu(\underbrace{DV\times \ldots \times DV}_{n \text{ times}}\times EV)\subseteq V\cdot (\underbrace{DV\times \ldots \times DV}_{n \text{ times}}\times EV),
$$
which settles the claim and thus the induction step.

(2)$\saa$(3): Suppose $U\ni 1$ is open, $n\geqslant 1$, $\eps>0$ and fix any $x\in X$. Let also $V=U\cap \{g\in G\del d(gx,x)<\eps/2\}$. Pick a finite set $F\subseteq G$ such that 
$$
\underbrace{G\times\ldots\times G}_{n \text { times}}=V\cdot \big(\underbrace{FV\times \ldots\times FV}_{n \text{ times}}\big).
$$
and set $A=\{(f_1x,\ldots, f_nx)\del f_i\in F\}\subseteq X^n$. Then if $(y_1,\ldots, y_n)\in X^n$, we can find $g_i\in G$ such that $d(y_i,g_ix)<\eps/2$ for all $i$. Also, there are $w, v_i\in V$ such that $g_i=wf_iv_i$, whence
\[\begin{split}
d(y_i,wf_ix)&\leqslant d(y_i,g_ix)+d(g_ix,wf_ix)\\
&\leqslant \eps/2+d(wf_iv_ix,wf_ix)\\
&=\eps/2+d(v_ix,x)\\
&< \eps/2+\eps/2.
\end{split}\]
In particular, $V\cdot A$ and hence also $U\cdot A$  is $\eps$-dense in $X^n$.

(3)$\saa$(4): Let $d$ be a compatible left-invariant metric on $G$ and let $X$ be the completion of $G$ with respect to $d$. Since the left-shift action of $G$ on itself is transitive, this action extends to a continuous action by isometries on $(X,d)$ with a dense orbit. Moreover, we can see $G$ as a closed subgroup of ${\rm Isom}(X,d)$, when the latter is equipped with the pointwise convergence topology. By (3), the action of $G$ on $X$ is approximately oligomorphic.

The implication (4)$\saa$(1) is implicit in the proof of Theorem 5.2 in \cite{OB}.
\end{proof}

We shall return to the microscopic properties of Roelcke precompact Polish groups later in Section \ref{microscopic}.


\subsection{Fixed point properties}\label{fixed points}
We shall now briefly consider the connection between the aforementioned boundedness properties and fixed point properties for affine actions on Banach spaces.
\begin{defi}
A Polish group $G$ has {\em property (ACR)} if any affine continuous action of $G$ on a separable reflexive Banach space has a fixed point. 
\end{defi}

\begin{prop}\label{OB then ACR}
Any Polish group with property (OB) has property (ACR).
\end{prop}

\begin{proof}
Assume $G$ has property (OB) and that  $\rho\colon G\til \Aff X$ is a continuous affine representation on a separable reflexive Banach space $X$ with linear part $\pi\colon G\til \GL X$ and associated cocycle $b\colon G\til X$.
Since $G$ has property (OB), the linear part $\pi$ is bounded, i.e., $\sup_{g\in G}\norm{\pi(g)}<\infty$, and we can therefore define a new equivalent norm $\triple\cdot$ on $X$ by
$$
\triple x=\sup_{g\in G}\norm{\pi(g)x},
$$
i.e., inducing the original topology on $X$. By construction, $\triple\cdot$ is $\pi(G)$-invariant and thus $\rho$ is an affine isometric representation of $G$ on $(X,\triple\cdot)$. By property (OB), every orbit $\rho(G)x\subseteq X$ is bounded and so, e.g., $C=\overline{\rm conv}(\rho(G)0)$ is a bounded closed convex set invariant under the affine action of $G$. As $X$ is reflexive, $C$ is weakly compact, and thus $G$ acts by affine isometries on the weakly compact convex set $C$ with respect to the norm $\triple\cdot$. It follows by the Ryll-Nardzewski fixed point theorem (Thm 12.22 \cite{fabian}) that $G$ has a fixed point on $X$. 
\end{proof}

Recall that a topological group $G$ is said to have {\em property (FH)} if any continuous affine isometric action on a Hilbert space has a fixed point, or, equivalently, has bounded orbits (see \cite{BHV}). So clearly property (ACR) is stronger than (FH). Similarly, fixed point properties for affine isometric actions on a Banach space $X$ have been studied for a variety of other classes of Banach spaces such as $L^p$  and uniformly convex spaces \cite{furman,cornulier}. It is worth noting that (ACR) characterises the compact groups within the class of locally compact Polish groups. This follows for example from the fact, proved by U. Haagerup and A. Przybyszewska \cite{haagerup}, that any locally compact Polish group $G$ admits a continuous affine isometric action on the strictly convex and reflexive space $\big(\bigoplus_nL^{2n}(G)\big)_2$ such that $\norm{g\cdot \xi-\xi}\Lim{g\til \infty}\infty$ for all $\xi \in \big(\bigoplus_nL^{2n}(G)\big)_2$. Also, as will be seen in Theorem \ref{loc comp equiv}, any non-compact locally compact Polish group $G$ admits a continuous affine (not necessarily isometric) action on a separable Hilbert space
with unbounded orbits.

A word of caution is also in its place with regards to continuous affine actions. While for an isometric action, either all orbits are bounded or no orbit is bounded, this is certainly not so for a general affine action on a Banach space. E.g., one can have a fixed point and still have unbounded orbits.

\begin{exa} There are examples of Polish groups that admit no non-trivial continuous representations in $\GL{\ku H}$ or even in $\GL{X}$, where $X$ is any separable reflexive Banach space. For example, note that the group of increasing homeomorphisms of $[0,1]$ with the topology of uniform convergence, $G={\rm Homeo}_+([0,1])$, has property (OB) (one way to see this is to note that the oligomorphic and hence Roelcke precompact group ${\rm Aut}(\Q,<)$ maps onto a dense subgroup of ${\rm Homeo}_+([0,1])$). Therefore, any continuous representation $\pi\colon G\til  \GL{X}$ must be bounded, whence
$$
\triple x=\sup_{g\in G}\norm{\pi(g)x}
$$
is an equivalent $G$-invariant norm on $X$. Though the norm may change, $X$ is of course still reflexive under the new norm and so $\pi$ can be seen as a strongly continuous linear isometric representation of $G$ on a reflexive Banach space. However, as shown by M. Megrelishvili \cite{megrelishvili}, any such representation is trivial, and so $\pi(g)=\Id$ for any $g\in G$.
\end{exa}

Though we have no example to this effect, the above example does seem to indicate that the class of reflexive spaces is too small to provide a characterisation of property (OB) and thus the implication (OB)$\saa$(ACR) should not reverse in general.


\subsection{Property (${\rm OB}_k$) and examples} We shall now consider a strengthening of property (OB) along with the class of Polish SIN groups.

\begin{defi}
Let $k\geqslant 1$. A Polish group $G$ is said to have {\em property (${\rm OB}_k$)} if whenever 
$$
W_0\subseteq W_1\subseteq W_2\subseteq\ldots \subseteq G
$$
is a exhaustive sequence of open subsets, then $G=W_n^k$ for some $n\geqslant 1$.
\end{defi}
Again, property (${\rm OB}_k$) has the following reformulation
\begin{itemize}
\item For any open symmetric $V\neq\tom$ there is a finite set $F\subseteq G$ such that $G=(FV)^k$.
\end{itemize}

Recall that a topological group is called a {\em SIN} group (for {\em small invariant neighbourhoods}) if it has a neighbourhood basis at the identity consisting of conjugacy invariant sets, or, equivalently, if the left and right uniformities coincide. For Polish groups, by a result of V. Klee, this is equivalent to having a compatible invariant metric, which necessarily is complete. Now, if $V\subseteq G$ is a conjugacy invariant neighbourhood of $1$, then $FV=VF$ for any set $F\subseteq G$, so we have the following set of equivalences.

\begin{prop}Let $G$ be a Polish SIN group.
Then the following conditions are equivalent.
\begin{enumerate}
\item $G$ is compact,
\item $G$ is Roelcke precompact,
\item $G$ has property ${\rm (OB}_k)$ for some $k\geqslant 1$.
\end{enumerate}
Also, the following conditions are equivalent.
\begin{enumerate}
\item $G$ is $\ku E_{ts}$-bounded,
\item $G$ is bounded,
\item $G$ is Roelcke bounded,
\item $G$ has property (OB).
\end{enumerate}
\end{prop}

\begin{proof}
The only non-trivial fact is that property ${\rm (OB}_k)$ implies compactness. But if $G$  has property ${\rm (OB}_k)$, then 
for any conjugacy invariant open set $V\ni 1$, there is a finite set $F\subseteq G$ such that $G=(VF)^k=V^kF^k$. So if $G$ is a SIN group, then for any open set $V\ni 1$ there is a finite set $F\subseteq G$ such that $G=VF$, i.e., $G$ is precompact. Since $G$ is Polish, it will actually be compact.
\end{proof}

\begin{exa} 
Let $E$ be the orbit equivalence relation induced by a measure pre\-ser\-ving ergodic automorphism of $[0,1]$ and let $[E]$ denote the corresponding {\em full group}, i.e., the group of measure-preserving automorphisms $T\colon [0,1]\til [0,1]$ such that $xET(x)$ for almost all $x\in [0,1]$, equipped with the invariant metric 
$$
d(T,S)=\lambda\big(\{x\in [0,1]\del T(x)\neq S(x)\}\big).
$$
Then $[E]$ is a non-compact,  Polish SIN group and, as shown in \cite{miller}, ${E}$ has property (OB). Thus, $[E]$ cannot have property (${\rm OB_k}$) for any $k\geqslant 1$.
\end{exa}

\begin{exa}
For another example, consider the separable commutative $C^*$-alge\-bra $\big(C(2^\N,\C), \norm\cdot_\infty\big)$
and the  closed multiplicative subgroup $C(2^\N,\T)$ consisting of all continuous maps from Cantor space $2^\N$ to the circle group $\T$.  So $C(2^\N,\T)$ is an abelian Polish group. Moreover, we claim that for any neighbourhood $V$ of the identity in $C(2^\N,\T)$, there is a $k$ such that any element $g\in C(2^\N,\T)$ can be written as $g=f^k$ for some $f\in V$. 

To see this, find some $k>1$ such that any continuous
$$
 f\colon 2^\N\til U=\{e^{2\pi i \alpha}\in\T\del -1/k<\alpha<1/k\;\}
$$ 
belongs to $V$. Fix $g\in C(2^\N,\T)$ and note that 
$$
A=\{x\in 2^\N\del g(x)\notin U\}
$$
and 
$$
B=\{x\in 2^\N\del g(x)=1\}
$$
are disjoint closed subsets of $2^\N$ and can therefore be separated by a clopen set $C\subseteq 2^\N$, i.e., $A\subseteq C\subseteq \;\sim\! B$. Then for any $x\in C$, writing $g(x)=e^{2\pi i\alpha}$ for $0<\alpha<1$, we set
$$
f(x)=e^{\frac{2\pi i\alpha}{k}}
$$
Also, if $x\notin C$, write $g(x)=e^{2\pi i \alpha}$ for some $-1/k<\alpha<1/k$ and set 
$$
f(x)=e^{\frac{2\pi i\alpha}{k}}.
$$
Thus, $f^k(x)=g(x)$ for all $x\in C$ and, since $C$ is clopen, $f$ is easily seen to be continuous.
Moreover, since $f$ only takes values in $U$, we see that $f\in V$, which proves the claim.

In particular, this implies that $C(2^\N,\T)$ is a bounded group. On the other hand, the usual Haar basis in $C(2^\N,\C)$, i.e., the functions $h_n\in  C(2^\N,\C)$ defined by $h_n(x)=1$ if the $n$th coordinate of $x$ is $1$ and $0$ otherwise, form an infinite family of distance $1$ from each other, so $C(2^\N,\T)$ is not compact.
\end{exa}

\begin{prob}As was shown in \cite{OB}, ${\rm Homeo}(S^2)$ has property (OB) and, as it has a connected open subgroup of finite index, it is actually bounded. On the other hand, it is an open problem due to V. Uspenski\u\i{} whether ${\rm Homeo}(S^2)$ is Roelcke precompact. In fact, it seems to be unknown even whether it has property  (${\rm OB_k}$) for any $k\geqslant 1$.
\end{prob}


\subsection{Non-Archimedean Polish groups}\label{non-archimedean}
Of special interest in logic are the automorphism groups of countable first order structures, that is, the closed subgroups of the group of all permutations of the natural numbers, $S_\infty$. Recall that the topology on $S_\infty$ is  the topology of pointwise convergence on $\N$ viewed as a discrete space. Thus, a neighbourhood basis at the identity in $S_\infty$ consists of the pointwise stabilisers (or {\em isotropy} subgroups) of finite subsets of $\N$, which are themselves subgroups of $S_\infty$. As is wellknown and easy to see, the property of having a neighbourhood basis at $1$ consisting of open subgroups actually isomorphically  characterises  the closed subgroups of $S_\infty$ within the class of Polish groups.

\begin{defi}
A Polish group $G$ is {\em non-Archimedean} if it has a neighbourhood basis at $1$ consisting of open subgroups. Equivalently, $G$ is non-Archimedean if it is isomorphic to a closed subgroup of $S_\infty$.
\end{defi}

\begin{prop}Let $G$ be a non-Archimedean Polish group.
Then 
\begin{enumerate}
\item $G$ is  Roelcke precompact if and only if it is Roelcke bounded,
\item $G$ is bounded if and only if it is compact.
\end{enumerate}
\end{prop}

\begin{proof}
Note that, for a closed subgroup $G$ of $S_\infty$, in the definition of Roelcke precompactness it suffices to quantify over open subgroups $V\leqslant G$. So $G$ is Roelcke precompact if and only if for any open subgroup $V\leqslant G$, the double coset space ${\rm V}\backslash {\rm G}\,\slash \,{\rm V}$ is finite, which is implied by Roelcke boundedness. As also Roelcke precompactness implies Roelcke boundedness, (1) follows.

For (2), it suffices to notice that if $G$ is not compact, then it has an open subgroup of infinite index and thus $G$ cannot be bounded. 
\end{proof}

\begin{exa} By Theorem 5.8 of \cite{OB}, the isometry group of the rational Urysohn metric space $\Q\U_1$ of diameter $1$ has property (OB) even as a discrete group and thus also as a Polish group. However, since it acts continuously and transitively on the discrete set $\Q\U_1$, but not oligomorphically, the group cannot be Roelcke precompact. 
\end{exa}

\begin{exa}
Note that if $G$ is a non-Archimedean Polish group, then $G$ is SIN if and only if $G$ has a neighbourhood basis at the identity consisting of {\em normal} open subgroups. For this, it suffices to note that if $U$ is a conjugacy invariant neighbourhood of $1$, then $\langle U\rangle$ is a conjugacy invariant open subgroup of $G$, i.e., $\langle U\rangle$ is normal in $G$.
In this case, we can find a decreasing series
$$
G\geqslant V_0\geqslant V_1\geqslant V_2\geqslant \ldots
$$
of normal open subgroups of $G$ forming a neighbourhood basis at $1$, and, moreover, this characterises the class of Polish, non-Archimedean, SIN groups. Note also that such a $G$ is compact if and only if all the quotients $G/V_n$ are finite. Moreover, as any countably infinite group admits a continuous affine action on a separable Hilbert space without fixed points (cf. Theorem \ref{loc comp equiv}), we see that $G$ is compact if and only if it has property (ACR).
\end{exa}

Of course, not all non-Archimedean Polish groups have non-trivial countable quotients, but, as we shall see, property (OB) can still be detected by actions on countable sets.

\begin{lemme}
Let $G$ be a non-Archimedean Polish group without property (OB). Then $G$ acts continuously and by isometries on a countable discrete metric space $(X,d)$ with unbounded orbits.
\end{lemme}

\begin{proof}
Since $G$ fails (OB) there is an open subgroup  $V\leqslant G$ such that for any finite subset $F\subseteq G$ and any number $k\geqslant 1$ we have $G\neq (VF)^k$. Let also $\{1\}=F_1\subseteq F_2\subseteq\ldots\subseteq G$ be an increasing sequence of finite symmetric subsets  whose union is dense in $G$ and set 
$$
B_k=(F_kVF_k)^{k!}.
$$
Then $B_k$ is symmetric, $1\in B_k\subseteq (B_k)^k\subseteq B_{k+1}$ and any two-sided coset $VaV$ is a subset of all but finitely many  $B_k$. 

Consider now the left coset space $G/V$ and define $\delta\colon G/V\til [1,\infty\,[$ by 
$$
\delta(gV,fV)=\min\big(k\del Vf\inv gV\subseteq B_k\big).
$$
Define also a metric $d$ on $G/V$ by $d(gV,gV)=0$ and 
$$
d(gV,fV)=\inf\big(\sum_{i=0}^{k-1}\delta(h_iV,h_{i+1}V)\del h_0V,h_1V,\ldots, h_kV \text{ is a path from }gV \text{ to }fV\big),
$$
whenever $gV\neq fV$. Clearly both $\delta$ and $d$ are invariant under the left-shift action by $G$ on $G/V$ and $d(gV,fV)\geqslant 1$ for any distinct $gV$ and $fV$.

Now assume towards a contradiction that $d(gV,V)\leqslant k$ for all $gV\in G/V$. Then certainly, by the definition of $d$, from any coset $gV$ there is a path $h_0V,h_1V,\ldots, h_kV$ ending at $h_kV=V$ with $\delta(h_iV,h_{i+1}V)\leqslant k$ for all $i$. It follows that
\[\begin{split}
g&\in VgV\\
&=Vh_k\inv h_0V\\
&\subseteq \big(Vh_k\inv h_{k-1}V\big)\cdot \big(Vh_{k-1}\inv h_{k-2}V\big)\cdot\ldots\cdot \big(Vh_1\inv h_0V\big)\\
&\subseteq \underbrace{B_k\cdot B_k\cdot\ldots\cdot B_k}_{k \text{ times}}\\
&\subseteq B_{k+1},
\end{split}\]
contradicting that $G\neq B_{k+1}$.

We have therefore constructed a transitive action of $G$ by isometries on the discrete metric space $(G/V,d)$ of infinite diameter. Moreover, as $V$ is a clopen subgroup of $G$, the action is continuous.
\end{proof}

\begin{thm}
Assume a Polish group $G$ acts continuously and by isometries on a countable discrete metric space $(X,d)$ with unbounded orbits.
Then $G$ admits an unbounded continuous linear representation $\pi\colon G\til \GL{\ku H}$ on a separable Hilbert space. 
\end{thm}

\begin{proof}
Suppose that $(X,d)$ is a countable discrete metric space on which $G$ acts continuously by isometries and fix some point $p\in X$. We define a function $\sigma$ on $X$ by setting 
$$
\sigma(x)=\min (k\geqslant 2\del d(x,p)\leqslant 2^k)
$$
and note that for any $g\in G$ and $x\in X$
$$
d(gx,p)\leqslant d(gx,gp)+d(gp,p)=d(x,p)+d(gp,p),
$$
and so
$$
\sigma(gx)\leqslant \max\{\sigma(gp),\sigma(x)\}+1
$$
and 
$$
\frac{\sigma(gx)}{\sigma(x)}\leqslant \sigma(gp).
$$

Let $\ell_2(X)$ denote the Hilbert space with orthonormal basis $(e_x)_{x\in X}$. For any $g\in G$, we define a bounded weighted shift $T_g$ of $\ell_2(X)$ by letting
$$
T_g(e_x)=\frac{\sigma(gx)}{\sigma(x)} e_{gx}
$$
and extending $T_g$ by linearity to the linear span of the $e_x$ and by continuity to all of $\ell_2(X)$. Note also that in this case $T_{g\inv }=T_{g}\inv$ and $T_{gf}=T_gT_f$, so the mapping $g\in G\til T_g\in GL(\ell_2(X))$ is a continuous representation.

Moreover, as
$$
\norm{T_g}\geqslant \frac{\sigma(gp)}{\sigma(p)}
$$
and $\frac{\sigma(gp)}{\sigma(p)}\til \infty$ as $d(gp,p)\til \infty$, we see that the representation $\pi\colon g\mapsto T_g$ is unbounded.
\end{proof}


In the following, if $V$ is an open subgroup of a Polish group $G$, we let $\pi_V\colon G\til {\rm Sym}(G/V)$ denote the continuous homomorphism induced by the left-shift action of $G$ on $G/V$, where ${\rm Sym}(G/V)$ is the Polish group of all permutations of $G/V$ with isotropy subgroups  of $gV\in G/V$ declared to be open. 

Also, if $V$ and $W$ are subgroups of a common group $G$, we note that 
$$
[W:W\cap V]=\text{ number of distinct left cosets of $V$ contained in $WV$}.
$$
In particular, for any $g\in G$,
\[\begin{split}
[V:V\cap V^g]
&=\text{ number of distinct left cosets of $V^g$ contained in $VV^g$}\\
&=\text{ number of distinct left cosets of $V$ contained in $VgV$}.
\end{split}\]
The {\em commensurator} of $V$ in $G$ is the subgroup
$$
{\rm Comm}_G(V)=\{g\in G\del [V:V\cap V^g]<\infty \text{ and } [V^g:V^g\cap V]=[V:V\cap V^{g\inv}]<\infty\},
$$
whereby $G={\rm Comm}_G(V)$ if and only if, for any $g\in G$, $VgV$ is a finite union of left cosets of $V$.

\begin{prop}\label{locally compact quotient}
The following are equivalent for a Polish group $G$.
\begin{enumerate}
\item  There is an open subgroup $W\leqslant G$ of infinite index such that $G={\rm Comm}_G(W)$,
\item there is an open subgroup $V\leqslant G$ of infinite index such that $\overline{\pi_V(G)}$ is a non-compact, locally compact subgroup of ${\rm Sym}(G/V)$.
\end{enumerate} 
\end{prop}

\begin{proof}
(1)$\saa$(2): Assume that (1) holds and let $V=W$. To see that $\overline{\pi_V(G)}$ is non-compact, just note that the $\pi_V(G)$-orbit of $1V\in G/V$ is infinite. For local compactness, let $H$ denote the isotropy subgroup of $1V$ in ${\rm Sym}(G/V)$. Since $H$ is clopen, it suffices to show that $H\cap \pi_V(G)$ is relatively compact, i.e., that any orbit of $\pi_V\inv(H)=V$ on $G/V$ is finite. But this follows directly from the fact that for any $gV\in G/V$, the set $VgV$ is a union of finitely many left cosets of $V$.

(2)$\saa$(1): Assume that (2) holds and find a neighbourhood $U$ of the identity in ${\rm Sym}(G/V)$ such that $U\cap \overline{\pi(G)}$ is compact. By decreasing $U$, we may suppose that $U=H_{g_1V}\cap\ldots\cap H_{g_nV}$, where $H_{g_iV}$ is the isotropy subgroup of $g_iV\in G/V$. Now, as $U$ is clopen, we have  $U\cap \overline{\pi_V(G)}=\overline{U\cap \pi_V(G)}$ and thus  $U\cap \overline{\pi_V(G)}$ is compact if and only if $U\cap \pi_V(G)$ has only finite orbits on $G/V$, i.e., 
$$
W=V^{g_1}\cap\ldots\cap V^{g_n}=\pi\inv(H_{g_1V}\cap \ldots\cap H_{g_nV})
$$
has only finite orbits on $G/V$. Moreover, this happens exactly when,  for all $f\in G$, $WfV$ is a finite union of left cosets of $V$, i.e., when $[W:W\cap V^f]<\infty$ for all $f\in G$.

Since the intersection of finitely many finite index subgroups of $W$  has finite index in $W$, it follows that for any $f\in G$,
\[\begin{split}
[W:W\cap W^f]&=[W:W\cap (V^{g_1}\cap \ldots\cap V^{g_n})^f]\\
&=[W:W\cap V^{fg_1}\cap \ldots\cap V^{fg_n}]\\
&=[W:(W\cap V^{fg_1})\cap \ldots\cap (W\cap V^{fg_n})]\\
&<\infty,
\end{split}\]
i.e., $G={\rm Comm}_G(W)$.
\end{proof}


\subsection{Locally compact groups}\label{locally compact}
We shall now turn our attention to the class of locally compact groups and see that the hierarchy of boundedness properties in Figure (1) collapses to just compactness and property (FH).
\begin{thm}\label{locally compact affine}
Let $G$ be a compactly generated, locally compact Polish group. Then $G$ admits an affine continuous representation  on separable Hilbert space $\ku H$ such that for all $\xi\in \ku H$,
$$
\norm{g \cdot \xi}\Lim{g\til \infty}\infty.
$$
\end{thm}

\begin{proof}Without loss of generality, $G$ is non-compact.
We fix a symmetric compact neighbourhood $V\subseteq G$ of $1$ generating $G$, i.e., such that 
$$
V\subseteq V^2\subseteq V^3\subseteq \ldots \subseteq G=\bigcup_{n\in \N}V^n.
$$
Moreover, by the non-compactness of $G$, we have $G\neq V^n$ for all $n$ and thus $B_n=V^n\setminus V^{n-1}\neq \tom$. Also, for any $m$ and $n$,
$$
B_n\cdot B_m\subseteq V^n\cdot V^m=V^{n+m}\subseteq B_1\cup \ldots\cup B_{n+m}.
$$
So, if $g\in B_n$, then for any $k$,
$$
gB_m\cap B_k\neq \tom \saa k\leqslant m+n.
$$
But, as also  $g\inv \in B_n$, we see that
$$
gB_m\cap B_k\neq\tom \saa B_m\cap g\inv B_k\neq \tom\saa m\leqslant k+n.
$$
In other words, for any $g\in B_n$,
$$
gB_m\cap B_k\neq \tom \saa m-n\leqslant k\leqslant m+n, 
$$
and thus, in fact, for any $m,n\geqslant 1$,
$$
B_n\cdot B_m\subseteq B_{m-n}\cup \ldots\cup B_{m+n}.
$$
Now, since $V^2$ is compact and ${\rm int} \;V\neq \tom$, there is a finite set $F\subseteq G$ such that $V^2\subseteq FV$, whence
$$
\lambda(B_n)\leqslant \lambda (V^n)\leqslant \lambda(F^{n-1}V)\leqslant |F|^{n-1}\lambda(V),
$$
where $\lambda$ is left Haar measure on $G$. Choosing $r>\max\{|F|,\lambda(V)\}$, we have $\lambda(B_n)\leqslant r^n$ for all $n\geqslant 1$.

Consider now the algebraic direct sum $\bigoplus_{n\geqslant 1}L_2(B_n,\lambda)$ equipped with the inner product
$$
\langle\xi\del\zeta\rangle=\sum_{n\geqslant 1}r^{2n}\int_{B_n}\xi\cdot \zeta\; d\lambda.
$$
So the completion $\ku H$ of $\bigoplus_{n\geqslant 1}L_2(B_n,\lambda)$ with respect to the corresponding norm $\norm\cdot$ consists of all Borel measurable functions $\xi\colon G\til \R$ satisfying 
$$
\norm{\xi}^2=\langle\xi\del \xi\rangle=\sum_{n\geqslant 1}r^{2n}\int_{B_n}\xi^2\; d\lambda<\infty
$$
and thus,  in particular, $\int_{B_n}\xi^2\;d\lambda\;\Lim{n}\; 0$.
Moreover, if we let $\zeta$ be defined by $\zeta\equiv \frac 1{r^{2n}}$ on $B_n$, then 
\[\begin{split}
\norm{\zeta}^2
&=\sum_{n\geqslant 1}r^{2n}\int_{B_n}\Big(\frac 1{r^{2n}}\Big)^2\; d\lambda   \\
&= \sum_{n\geqslant 1}r^{2n}\lambda(B_n)\frac 1{r^{4n}}   \\
&\leqslant \sum_{n\geqslant 1}\frac 1 {r^n}\\
&<\infty.
\end{split}\]
So $\zeta\in \ku H$.
Also, for any $\xi \in \ku H$, 
\[\begin{split}
\langle\xi\del \zeta\rangle
&=\sum_{n\geqslant 1}r^{2n}\int_{B_n}\xi\frac 1{r^{2n}}\; d\lambda   \\
&= \sum_{n\geqslant 1}\int_{B_n}\xi\; d\lambda  \\
&=\int_{G}\xi\; d\lambda.
\end{split}\]
Let $\ku H_0$ denote the orthogonal complement of $\zeta$ in $\ku H$, i.e., $\ku H_0=\{\xi\in \ku H\del \int_G\xi\; d\lambda=0\}$.

We now define $\pi\colon G\til \GL{\ku H}$ to be the left regular representation, i.e., $\pi(g)\xi=\xi(g\inv\,\cdot\,)$. To see that this is well-defined, that is, that each $\pi(g)$ is a bounded operator on $\ku H$, note that if $g\in B_m$ and $\xi\in \ku H$, then 
\[\begin{split}
\norm{\pi(g)\xi}^2
&=\sum_{n\geqslant 1}r^{2n}\int_{B_n}\xi(g\inv \,\cdot\,)^2\; d\lambda   \\
&=\sum_{n\geqslant 1}r^{2n}\int_{g\inv B_n}\xi^2\; d\lambda   \\
&=\sum_{n\geqslant 1}r^{2n}\sum_{k\geqslant 1}\int_{g\inv B_n\cap B_k}\xi^2\; d\lambda   \\
&=\sum_{n,k\geqslant 1}r^{2n}\int_{g\inv B_n\cap B_k}\xi^2\; d\lambda  \\
&\leqslant\sum_{n,k\geqslant 1}r^{2(m+k)}\int_{g\inv B_n\cap B_k}\xi^2\; d\lambda  \\
&=2^{2m}\sum_{n,k\geqslant 1}r^{2k}\int_{g\inv B_n\cap B_k}\xi^2\; d\lambda  \\
&=2^{2m}\sum_{k\geqslant 1}r^{2k}\int_{B_k}\xi^2\; d\lambda  \\
&=2^{2m}\norm{\xi}^2.
\end{split}\]
Thus, $\norm{\pi(g)}\leqslant 2^m$ for all $g\in B_m$, showing also that $\norm{\pi(g)}$ is uniformly bounded for $g$ in a neighbourhood of $1\in G$.

Finally, as 
$$
\langle\pi(g)\xi\del \zeta\rangle=\int_G\xi(g\inv\,\cdot\,)\;d\lambda=\int_G\xi\;d\lambda=\langle\xi\del \zeta\rangle,
$$
we see that $\ku H_0$ is $\pi(G)$-invariant.

Now define a $\pi$-cocycle $b\colon G\til \ku H_0$ by $b(g)=\pi(g)\chi_V-\chi_V$, where $\chi_V$ is the characteristic function of $V\subseteq G$, and let $\rho\colon G\til \Aff{\ku H}$ be the corresponding affine representation, $\rho(g)\xi=\pi(g)\xi+b(g)$. Thus,
$$
\rho(g)\xi-\xi=\xi(g\inv\,\cdot\,)+\chi_V(g\inv\,\cdot\,)-\xi-\chi_V
$$
for any $g\in G$ and $\xi\in \ku H$.

\begin{claim}For all $\xi\in \ku H_0$ and $K$, there is a compact set $C\subseteq G$ such that
$$
\norm{\rho(g)\xi-\xi}> K
$$
for all $g\in G\setminus C$.
\end{claim}

Assume first that $\xi\equiv-1 $ on $V$. Then, as $\xi\in \ku H_0$, we have $\int_G\xi d\lambda=0$ and thus there must be a Borel set $E\subseteq G$ such that $\lambda(E)>\delta>0$ and $\xi>\delta$ on $E$. Also, without loss of generality, we may assume that $E\subseteq B_n$ for some $n>1$. 
Since $\int_{B_m}\xi^2d\lambda\Lim{m}0$, we can find $M>n+2$ such that for all $m> M$, we have  $r^{2(m-n)}\frac{\delta^3}8>K$ and the set 
$$
F_m=\big\{h\in B_{m-n}\cup\ldots\cup B_{m+n}\del |\xi(h)|>\frac \delta2\big\}
$$
has measure $<\frac \delta2$.

Thus, if $g\in B_m$, $m>M$, we have $gE\subseteq B_{m-n}\cup\ldots\cup B_{m+n}$ and so $gE\cap V=\tom$, while also $E\cap V=\tom$. Therefore, 
\[\begin{split}
\norm{\rho(g)\xi-\xi}^2
&\geqslant r^{2(m-n)}\int_{B_{m-n}\cup \ldots\cup B_{m+n}}\big(\rho(g)\xi-\xi)^2d\lambda\\
&\geqslant r^{2(m-n)}\int_{gE\setminus F_m}\big(\xi(g\inv\,\cdot\,)+\chi_V(g\inv\,\cdot\,)-\xi-\chi_V\big)^2d\lambda\\
&\geqslant r^{2(m-n)}\int_{gE\setminus F_m}\big(\xi(g\inv\,\cdot\,)-\xi\big)^2d\lambda\\
&\geqslant r^{2(m-n)}\frac {\delta^3}8\\
&>K.
\end{split}\]
Setting $C=B_1\cup \ldots \cup B_M=V^M$, the claim follows.

Assume now instead that $\xi\not\equiv-1 $ on $V$ and  fix $A\subseteq V$  a Borel set of positive measure $\lambda(A)>\eps>0$ such that $|\xi+1|>\eps$ on $A$. As $\int_{B_n}\xi^2\;d\lambda\Lim n 0$, there is an $N>2$ such that, for all $m>N$, the set 
$$
D_m=\Big\{h\in B_{m-1}\cup B_m\cup B_{m+1}\Del |\xi(h)|>\frac {\eps} 2\Big\}
$$
has measure $<\frac \eps2$.

In particular, if $g\in B_m$ for some $m>N$, then $gA\subseteq B_mB_1\subseteq B_{m-1}\cup B_m\cup B_{m+1}$ and so for any $h\in gA\setminus D_m$, as $gA\cap V=\tom$,
$$
|\rho(g)\xi(h)-\xi(h)|=|\xi(g\inv h)+\chi_V(g\inv h)-\xi(h)-\chi_V(h)|=|\xi(g\inv h)+1-\xi(h)-0|\geqslant \frac \eps2.
$$
It follows that for such $g$,
\[\begin{split}
\norm{\rho(g)\xi-\xi}^2
&\geqslant r^{2m-2}\int_{B_{m-1}\cup B_m\cup B_{m+1}}\big(\rho(g)\xi-\xi)^2d\lambda\\
&\geqslant r^{2m-2}\int_{gA\setminus D_m}\big(\frac \eps2\big)^2d\lambda\\
&\geqslant r^{2m-2}\frac {\eps^3}8.
\end{split}\]
Choosing $M>N$ large enough such that $ r^{2M-2}\frac {\eps^3}8>K$, we see that for all $g\notin C=B_1\cup \ldots \cup B_M=V^M$, we have  $\norm{\rho(g)\xi-\xi}^2>K$, proving the claim and thus the theorem.
\end{proof}

Let us also mention that the preceding theorem holds for all locally compact second countable (i.e., Polish locally compact) groups. For, in order to extend the argument above to the non-compactly generated groups, it suffices to produces a covering $V_0\subseteq V_1\subseteq \ldots \subseteq G$ by compact subsets such that $V_n\cdot V_m\subseteq V_{n+m}$, which can be done, e.g., by Theorem 5.3 of \cite{haagerup}. However, we shall not need this extension as any such $G$ admits a fixed point free affine isometric action on Hilbert space.

We can now state the following equivalences for locally compact Polish groups.
\begin{thm}\label{loc comp equiv}
Let $G$ be a locally compact Polish group and $\ku H$ be a separable Hilbert space. Then the following are equivalent.
\begin{enumerate}
\item $G$ is compact,
\item $G$ has property (OB),
\item $G$ has property (ACR),
\item any continuous linear representation $\pi\colon G\til \GL{\ku H}$  is bounded,
\item any continuous affine representation $\pi\colon G\til \Aff{\ku H}$  fixes a point,
\item any continuous affine representation $\pi\colon G\til \Aff{\ku H}$  has  a bounded orbit.
\end{enumerate}
\end{thm}

\begin{proof}
Clearly, (1) implies (2) and, by Proposition \ref{OB then ACR} (2) implies (3). Moreover, by Theorem \ref{OB}, (2) implies (4). Since (5) trivially implies (6), it suffices to show that (4) and (6) each imply (1). 

To see that (6) implies (1), note first that if $G$ is not compactly generated, then $G$ can be written as the union of an increasing chain of proper open subgroups, in which case it is well-known that $G$ admits a continuous affine isometric representation on a separable Hilbert space with unbounded orbits (see Corollary 2.4.2 \cite{BHV}). On the other hand, if $G$ is compactly generated, it suffices to apply Theorem \ref{locally compact affine}.

Finally, to see that (4) implies (1), assume that $G$ is locally compact, non-compact and let $\lambda$ be left-Haar measure on $G$. Since $G$ is $\sigma$-compact, we can find an exhaustive sequence
$$
A_0\subseteq A_1\subseteq A_2\subseteq \ldots\subseteq G
$$
of compact neighbourhoods of the identity. Set $B_k=(A_k)^{k!}\setminus (A_{k-1})^{(k-1)!}$, which is also relatively compact, and note that for any $g\in B_k$,
$$
gB_m\subseteq B_0\cup\ldots\cup B_{\max\{k,m\}+1},
$$
whence $B_m\cap g\inv B_l=\tom$ for all $l>\max\{k,m\}+1$. Moreover, since $A_0$ has non-empty interior and every $A_k$ is compact, it is easy to see that similarly $B_k$ has non-empty interior and thus $\lambda(B_k)>0$.

We now note that for any $g\in B_k$, the sets 
$$
\{B_m\cap g\inv B_l\}_{l,m}\qquad \text{and}\qquad \{gB_m\cap  B_l\}_{l,m}
$$
each form Borel partitions of $G$ and so $L_2(G)$ can be orthogonally decomposed as
$$
L_2(G)=\bigoplus_{l,m}L_2\big(B_m\cap g\inv B_l\big)=\bigoplus_{l,m}L_2\big(gB_m\cap B_l\big).
$$
Moreover, for every pair $m,l$, we can define an isomorphism 
$$
T_g^{m,l}\colon L_2\big(B_m\cap g\inv B_l\big)\til L_2\big(gB_m\cap B_l\big)
$$ 
by 
$$
T_g^{m,l}(f)=\exp(l-m)f(g\inv \;\cdot\;)
$$
and note that $\norm{T_g^{m,l}(f)}_2=\exp(l-m)\norm{f}_2$ for any $f\in L_2\big(B_m\cap g\inv B_l\big)$. Since $L_2\big(B_m\cap g\inv B_l\big)=\{0\}$, whenever $l>\max\{k,m\}+1$, it follows that the linear operator
$$
T_g=\bigoplus_{m,l}T_g^{m,l}\colon L_2(G)\til L_2(G)
$$
is well-defined, invertible and $\norm{T_g}\leqslant \exp({k+1})$.

So $g\mapsto T_g$ defines a continuous representation of $G$ in $\GL{L_2(G)}$. To see that it is unbounded, note that for any $m$ there are arbitrarily large $l$ such that $\lambda(B_m\cap g\inv B_l)>0$ for some $g\in G$ and so
$$
\Norm{T_g}\geqslant 
\frac{\norm{T_g(\chi_{B_m\cap g\inv B_l})}_2}
{\norm{\chi_{B_m\cap g\inv B_l}}_2}
=\exp(l-m).
$$
Since $\exp(l-m)\Lim{l\til \infty}\infty$, we see that the representation is unbounded.
\end{proof}

The following corollary is now immediate by Proposition \ref{locally compact quotient}.
\begin{cor}
Suppose $G$ is a Polish group and $W\leqslant G$ is an open subgroup of infinite index with $G={\rm Comm}_G(W)$. Then $G$ admits a continuous affine representation  on a separable Hilbert space for which every orbit is unbounded.
\end{cor}


\section{Local boundedness properties}
Having studied the preceding global boundedness properties for Polish groups, it is natural to consider their {\em local} counterparts, where by this we understand the existence of a neighbourhood $U\subseteq G$ of the identity satisfying similar covering properties to those listed in Figure 1.

\subsection{A question of Solecki}As mentioned earlier, in \cite{actions} and \cite{uspenskii}, S. Solecki and V. Uspenski\u\i{} independently showed that a Polish group $G$ is compact if and only if  for any open set $V\ni 1$ there is a  finite set $F\subseteq G$ with $G=FVF$. The similar characterisation of compactness with only one-sided translates $FV$ on the other hand is fairly straightforward.

A similar characterisation of locally compact Polish groups is also possible, namely, a Polish group $G$ is locally compact if and only if there is an open set $U\ni1$ such that for any open $V\ni 1$ there is a finite set $F\subseteq G$ such that $U\subseteq FV$.

The corresponding property for two-sided translates leads to the following definition.

\begin{defi}
A topological group $G$ is {\em  feebly locally compact} if there is a neighbourhood $U\ni1$ such that for any open $V\ni1 $ there is a finite set $F\subseteq G$ satisfying $U\subseteq FVF$.
\end{defi}

Solecki \cite{haar null} originally considered groups in the complement of this class and termed these {\em strongly non-locally compact} groups. In connection with Solecki's \cite{haar null} study of left-Haar null sets in Polish groups  \cite{haar null}, the class of strongly non-locally compact Polish groups turned out to be of special significance when coupled with the following concept.

\begin{defi}[S. Solecki \cite{haar null}]
A Polish group $G$ is said to have a {\em free subgroup at $1$} if there is a sequence $g_n\in G$ converging to $1$ which is the basis of a free non-Abelian group and such that any finitely generated subgroup $\langle g_1,\ldots,g_n\rangle$ is discrete.
\end{defi}

Solecki asked whether any (necessarily non-locally compact) Polish group having a free subgroup at $1$ can be feebly locally compact (Question 5.3 in \cite{haar null}).  We shall now present a fairly general construction of Polish groups that are feebly locally compact, but nevertheless fail to be locally compact. Depending on the specific inputs, this construction also provides an example with a free subgroup at $1$ and hence an answer to Solecki's question.

\subsection{Construction}
Fix a countable  group $\Gamma$ and let 
$$
H_\Gamma=\{g\in \Gamma^\Z\del \e m\;\a n\geqslant m\; g(n)=1\},
$$
which is a subgroup of the full direct product $\Gamma^\Z$. Though $H_\Gamma$ is not closed in the product topology for $\Gamma$ discrete, we can equip $H_\Gamma$ with a complete $2$-sided invariant ultrametric $d$ by the following definition.
$$
d(g,f)=2^{\max \big(k\del g(k)\neq f(k)\big)}.
$$
By the definition of $H_\Gamma$, this is well-defined and it is trivial to see that the ultrametric inequality
$$
d(g,f)\leqslant \max\big\{d(g,h),d(h,f)\big\}
$$
is verified. Also, since
$$
\max \big(k\;|\: h(k)g(k)\neq h(k)f(k)\big)=\max \big(k\;|\: g(k)\neq f(k)\big)=\max \big(k\;|\: g(k)h(k)\neq f(k)h(k)\big),
$$
we see that the metric is $2$-sided invariant and hence induces a group topology on $H_\Gamma$. Moreover, the countable set 
$$
\big\{g\in H_\Gamma\del \{k\del g(k)\neq 1\}\text{ is finite }\big\}
$$
is dense in $H_\Gamma$, so $H_\Gamma$ is separable and is easily seen to be complete, whence $H_\Gamma$ is a Polish group. To avoid confusion with the identity in $\Gamma$, denote by $e$ the identity in $H_\Gamma$, i.e., $e(n)=1$ for all $n\in \Z$.

We now let $\Z$ act by automorphisms on $H_\Gamma$ via bilateral shifts of sequences, that is, for any $k\in \Z$ and $g\in H_\Gamma$, we let
$$
\big(k* g\big)(n)=g(n-k)
$$
for any $n\in \Z$. In particular, for any $g,f\in H_\Gamma$ and $k\in \Z$, we have
$$
d(k* g,k* f)=2^kd(g,f),
$$
i.e., $k*B_d(e,2^m)=B_d(e,2^{m+k})$, 
which shows that $\Z$ acts continuously on $H_\Gamma$. We can therefore form the topological semidirect product
$$
\Z\ltimes H_\Gamma,
$$
which is just $\Z\times H_\Gamma$ equipped with the product topology and the group operation  defined by 
$$
(n,g)\cdot(m,f)=(n+m,g(n*f)).
$$ 
So $\Z\ltimes H_\Gamma$ is a Polish group. Also, a neighbourhood basis at the identity is given by the clopen subgroups
$$
V_m=\big\{(0,g)\in \Z\ltimes H_\Gamma\del \a i\geqslant m\; g(i)=1\big\}=\{0\}\times B_d(e, 2^{m}),
$$
which implies that $\Z\ltimes H_\Gamma$ is isomorphic to a closed subgroup of the infinite symmetric group $S_\infty$. Note also that
$$
\ldots\subseteq V_{-1}\subseteq V_0\subseteq V_1\subseteq \ldots.
$$
Then one can easily verify that for any $k\in \Z$ and $g\in H_\Gamma$
$$
(k,e)\cdot(0,g)\cdot(k,e)\inv=(0,k*g)
$$
and hence
$$
(k,e)\cdot V_m\cdot (k,e)\inv =  (k,e)\cdot\big(\{0\}\times B_d(e, 2^{m})\big)\cdot(k,e)\inv  =\{0\}\times   B_d(e, 2^{m+k})      =           V_{m+k}
$$
for any $k,m\in \Z$.

We claim that $\Z\ltimes H_\Gamma$ is Weil complete. To see this, suppose that $f_n\in \Z\ltimes H_\Gamma$ is left-Cauchy, i.e., that $f_n\inv f_m\Lim{n,m\til \infty}1$. Writing $f_n=(k_n,g_n)$ for $k_n\in \Z$ and $g_n\in H_\Gamma$, we have $f_n\inv=\big(-k_n,(-k_n)*g_n\inv\big)$ and so
\begin{align*}
f_n\inv f_m=&\big(-k_n,(-k_n)*g_n\inv\big)\big(k_m,g_m\big)\\
=&\big (k_m-k_n, ((-k_n)*g_n\inv)((-k_n)*g_m)\big)\\
=&\big (k_m-k_n, (-k_n)*(g_n\inv g_m)\big).
\end{align*}
Since  $f_n\inv f_m\Lim{n,m\til \infty}1$, the sequence $k_n\in \Z$ is eventually constant, say $k_n=k$ for $n\geqslant N$, and so for all $n,m\geqslant N$, 
\begin{align*}
f_n\inv f_m=\big (0, (-k)*(g_n\inv g_m)\big).
\end{align*}
Since $\Z$ acts continuously on $H_\Gamma$ it follows that $(-k)*(g_n\inv g_m)\Lim{n,m\til \infty}e$ if and only if $g_n\inv g_m\Lim{n,m\til \infty}e$, i.e., if and only if $(g_n)$ is left-Cauchy in $H_\Gamma$. Since $H_\Gamma$ has a complete $2$-sided invariant metric it follows that $(g_n)$ converges to some $g\in H_\Gamma$ and so $(f_n)$ converges in $\Z\ltimes H_\Gamma$ to $(k,g)$, showing that $\Z\ltimes H_\Gamma$ is Weil complete.

Denoting by $\F_\infty$ the free non-Abelian group on infinitely many letters $a_1,a_2,\ldots$, the following provides an easy answer to Solecki's question mentioned above.

\begin{thm}
The group $\Z\ltimes H_{\F_\infty}$ is a non-locally compact, Weil complete Polish group, having a free subgroup at $1$. Also, $\Z\ltimes H_{\F_\infty}$ is isomorphic to a closed subgroup of $S_\infty$ and there is an open subgroup $U\leqslant \Z\ltimes H_{\F_\infty}$ whose conjugates $fUf\inv$ provide a neighbourhood basis at $1$. In particular, $\Z\ltimes H_{\F_\infty}$  feebly locally compact.
\end{thm}

\begin{proof}
To see that $\Z\ltimes H_{\F_\infty}$ is not locally compact, we define for every $m\in \Z$ a continuous homomorphism
$$
\pi_m\colon\Z\ltimes H_{\F_\infty}\til \F_\infty
$$
by $\pi_m(k,g)=g(m)$. Keeping the notation from before, $V_m=\{0\}\times B_d(e,2^m)$, we see that for any $m\in \Z$, $\pi_{m-1}\colon V_{m}\til \F_\infty$ is surjective. So no $V_{m}$ is compact and hence $\Z\ltimes H_{\F_\infty}$ cannot be locally compact. Now, to see that $\Z\ltimes H_{\F_\infty}$ has a free subgroup at $1$, define $g_n\in H_\Gamma$ by
$$
g_n(k)=\begin{cases}
a_n, &\text{if } k\leqslant -n;\\
1, &\text{if }k>-n.
\end{cases}
$$
Then $(0,g_n)\Lim{n\til \infty}(0,e)$ in $\Z\ltimes H_{\F_\infty}$, so if we let $\beta_n=(0,g_n)$, we see that  $\langle \beta_1,\beta_2,\beta_3,\ldots\rangle$ is a non-discrete subgroup of $\Z\ltimes H_{\F_\infty}$. To see that $\{\beta_1,\beta_2,\beta_3,\ldots\}$ is a free basis for $\langle \beta_1,\beta_2,\beta_3,\ldots\rangle$, it suffices to see that for every $n$, $\langle \beta_1,\beta_2,\ldots, \beta_n\rangle$ is freely generated by $\{\beta_1,\beta_2,\ldots, \beta_n\}$. But this follows easily from the fact that $\pi_{-n}(\beta_i)=a_i$
for any $i\leqslant n$ and that $\pi_{-n}$ is a homomorphism into $\F_\infty$. This argument also shows that $\langle\beta_1,\beta_2,\ldots,\beta_n\rangle$ is discrete. So $\Z\ltimes H_{\F_\infty}$ has a free subgroup at $1$.

That $\Z\ltimes H_{\F_\infty}$ is isomorphic to a closed subgroup of $S_\infty$ has already been proved and, moreover, we know that for any $m$, $(m,e)\cdot V_{-m}\cdot (m,e)=V_0$. So for the last statement it suffices to take $U=V_0$.
\end{proof}



\section{Microscopic structure}\label{microscopic}

The negative answer to Solecki's question of whether feebly locally  compact Polish groups are necessarily locally compact indicates that there is a significant variety of behaviour in Polish groups with respect to coverings by translates of open sets. Also, as indicated by Malicki's result \cite{malicki} that no oligomorphic closed subgroup of $S_\infty$ is  feebly locally compact,  the macroscopic or large scale structure of Polish groups has counterparts at the local level. We shall now develop this even further by bringing it to the microscopic level of Polish groups as witnessed by neighbourhood bases at $1$. 

More precisely, we shall study the conditions under which any neighbourhood basis at $1$
$$
V_0\supseteq V_1\supseteq V_2\supseteq \ldots \ni 1
$$
will cover $G$ via, e.g.,  sequences of two-sided translates $G=\bigcup_{n\in \N}f_nV_ng_n$ by single elements $f_n,g_n\in G$ or two-sided translates $G=\bigcup_{n\in \N}F_nGE_n$ by  finite subsets $F_n, E_n\subseteq G$. As opposed to this, neighbourhood bases $V_0\supseteq V_1\supseteq V_2\supseteq \ldots \ni 1$ failing these covering properties can be thought of a being {\em narrow} in the group $G$.

As is fairly easy to see (see Proposition \ref{loc comp} below), some of these covering properties are generalisations of local compactness, but (which is less trivial)  must fail in non-compact Roelcke precompact Polish groups. Moreover, as will be shown in Section \ref{isometric actions}, the existence of narrow sequences $(V_n)$ allows for the construction of isometric actions on various spaces with interesting local dynamics.


\subsection{Narrow sequences and completeness}
To commence our study, let us first note how some of the relevant covering properties play out in the context of locally compact groups.

\begin{prop}\label{loc comp}
\begin{itemize}
\item[(a)] Suppose $G$ is a non-discrete, locally compact Polish group. Then there are open sets
$$
V_0\supseteq V_1\supseteq \ldots\ni 1
$$
such that for any $f_n\in G$, $G\neq \bigcup_n f_nV_n$.

\item[(b)] Suppose $G$ is a non-discrete, unimodular, locally compact Polish group. Then there are open sets
$$
V_0\supseteq V_1\supseteq \ldots\ni 1
$$
such that for any $f_n, g_n\in G$, $G\neq \bigcup_n f_nV_ng_n$.

\item[(c)] Suppose $G$ is a non-discrete, locally compact Polish group. Then for any open sets
$$
V_0\supseteq V_1\supseteq \ldots\ni 1
$$
there are finite sets $F_n\subseteq  G$ such that $G=\bigcup_n F_nV_n$.

\item [(d)] Let $\Gamma$ be a non-trivial finite group. Then $\Z\ltimes H_\Gamma$ is a non-discrete, locally compact Polish group having a compact open subgroup $U$ such that for any open set $V\ni 1$ there is $f\in \Z\ltimes H_\Gamma$ with $U\subseteq fVf\inv$. In particular, whenever 
$$
V_0\supseteq V_1\supseteq \ldots \ni1
$$ 
are open sets there are $f_n,g_n\in \Z\ltimes H_\Gamma$ such that $\Z\ltimes H_\Gamma=\bigcup_nf_nV_ng_n$.

\end{itemize}
\end{prop}

\begin{proof}
(a) Let $\lambda$ be left Haar measure on $G$ and choose $V_n\ni 1$ open such that
$\lambda(V_n)<\lambda(G)/2^{n+2}$. Then for any $f_n\in G$,
$$
\lambda(\bigcup_{n=0}^\infty f_nV_n)\leqslant \sum_{n=0}^\infty \lambda(f_nV_n)=\sum_{n=0}^\infty \lambda(V_n)<\lambda(G),
$$
so $G\neq \bigcup_nf_nV_n$.

(b) This is proved in the same manner as (a) using that Haar measure is $2$-sided invariant.

(c) Let $U\subseteq G$ be any compact neighbourhood of $1$. Then any open $V\ni 1$ covers $U$ by left-translates and hence by a finite number of left translates. So if open $V_0\supseteq V_1\supseteq\ldots\ni 1$ are given, find finite $E_n\subseteq G$ such that $U\subseteq E_nV_n$. Then if $\{h_n\}_{n\in \N}$ is a dense sequence in $G$, we have $G=\bigcup_{n}h_nU=\bigcup_{n}h_nE_nV_n$. Setting $F_n=h_nE_n$ we have the desired conclusion.

(d) One easily sees from the construction of $H_\Gamma$ that $B_d(e,1)$ is compact and thus $U=\{0\}\times B_d(e,1)$ is a compact neighbourhood of $1$ in $\Z\ltimes H_\Gamma$. So $\Z\ltimes H_\Gamma$ is locally compact. Now, if $V_0\supseteq V_1\supseteq\ldots \ni 1$ are open, there are $f_n\in \Z\ltimes H_\Gamma$ such that $U\subseteq f_nV_nf_n\inv$. So if $\{h_n\}_{n\in \N}$ is a dense subset of $\Z\ltimes H_\Gamma$, then 
$\Z\ltimes H_\Gamma=\bigcup_nh_nf_nV_nf_n\inv$.
\end{proof}

Our first task is to generalise Proposition \ref{loc comp} (a) to all non-discrete Polish groups.

\begin{prop}\label{left-translates}
Suppose $G$ is a non-discrete Polish group. Then there are open sets 
$$
V_0\supseteq V_1\supseteq \ldots \ni 1
$$
such that for any $g_n\in G$, $G\neq \bigcup_n g_n V_n$.
\end{prop}

\begin{proof}
Fix a compatible complete metric $d$ on $G$. We will inductively define symmetric open $V_0\supseteq V_1\supseteq \ldots\ni 1$ and $f_s\in G$ for $s\in 2^{<\N}\setminus \{\tom \}$ such that for any $s\in 2^{n}$ and $i\in \{0,1\}$ 
\begin{itemize}
\item ${\rm diam}(f_{si}\overline V_n)<\frac1{n+1}$,
\item $f_{s0}\overline V_n^2\cap f_{s1}\overline V_n^2=\tom$,
\item $f_{si}\overline V_n\subseteq f_sV_{n-1}$.
\end{itemize}
To see how this is done, begin by choosing $f_0, f_1\in G$ distinct and find $V_0\ni 1$ symmetric open such that $f_0\overline V_0^2\cap f_1\overline V_0^2=\tom$ and ${\rm diam}(f_i\overline V_0)<1$.
Now, suppose $V_{n-1}$ and $f_s$ are defined for all $s\in 2^n$. Then for every $s\in 2^n$, we choose distinct $f_{s0},f_{s1}\in f_sV_{n-1}$ and can then find $1\in V_n\subseteq V_{n-1}$ small enough such that the three properties hold.

Suppose that construction has been done and that $g_n\in G$ are given. We show that $G\neq \bigcup_n g_nV_n$ as follows. First, as $f_0\overline V_0^2\cap f_1\overline V_0^2=\tom$, there must be some $i_0\in \{0,1\}$ such that $f_{i_0}\overline V_0\cap g_0V_0=\tom$. And again, since $f_{i_00}\overline V_1\cap f_{i_01}\overline V_1=\tom$, there must be some $i_1\in \{0,1\}$ such that $f_{i_0i_1}\overline V_1\cap g_1V_1=\tom$. Etc. So inductively, we define $i_0,i_1,\ldots\in \{0,1\}$ such that for any $n$, $f_{i_0i_1\ldots i_n}\overline V_n\cap g_nV_n=\tom$. Since the $f_{i_0i_1\ldots i_n}\overline V_n$ are nested and have vanishing diameter it follows that $\bigcap_nf_{i_0i_1\ldots i_n}\overline V_n$ is non-empty and evidently disjoint from $\bigcup_ng_nV_n$.
\end{proof}

Topological groups $G$ with the property that for any open $V_0\supseteq V_1\supseteq \ldots \ni 1$ there are finite sets $F_n\subseteq G$ with $G=\bigcup_{n}F_nV_n$ are said to be {\em o-bounded} or {\em Menger bounded} \cite{tkachenko}. By Proposition \ref{loc comp} (c), these are clearly a generalisation of the locally compact Polish groups, but it remains open whether they actually coincide with the locally compact groups within the class of Polish groups. 
\begin{prob}\label{o-bounded}
Suppose $G$ is a non-locally compact Polish group. Are there open sets $V_0\supseteq V_1\supseteq \ldots \ni 1$ such that for any finite sets $F_n\subseteq G$, we have $G\neq \bigcup_n F_n V_n$?
\end{prob}
There is quite a large literature on o-boundedness in the context of general topological groups, though less work has been done on the more tractable subclass of Polish groups. T. Banakh \cite{banakh1,banakh2} has verified Problem \ref{o-bounded} under additional assumptions, one of them being Weil completeness \cite{banakh1}. We include a proof of his result here, as it can be gotten by only a minor modification of the proof of Proposition \ref{left-translates}.

\begin{prop}\label{Weil complete}
Suppose $G$ is a Weil complete, non-locally compact Polish group. Then there are open sets 
$$
V_0\supseteq V_1\supseteq\dots\ni 1
$$
such that for any finite subsets $F_n\subseteq G$, $G\neq \bigcup_nF_nV_n$.
\end{prop}

\begin{proof}
By the same  inductive procedure as before, we choose open $V_n\ni 1$ and $f_s\in G$ for every $s\in \N^{<\N}\setminus \{\tom\}$ such that for any $s\in \N^n$ and $i\neq j\in \N$,
\begin{itemize}
\item ${\rm diam}(f_{si}\overline V_n)<\frac1{n+1}$,
\item $f_{si}\overline V_n^2\cap f_{sj}\overline V_n^2=\tom$,
\item $f_{si}\overline V_n\subseteq f_sV_{n-1}$.
\end{itemize}
This is possible by the left-invariance of $d$ and the fact that no open set $V\ni 1$ is relatively compact and hence contains infinitely many disjoint translates of some open set $1\in U\subseteq V$. The remainder of the proof follows that of Proposition \ref{left-translates}.
\end{proof}

For good order, let us also state the analogue of Proposition \ref{Weil complete} for $2$-sided translates.

\begin{prop}\label{two-sided invariant}
Suppose $G$ is a  non-locally compact Polish SIN group. Then there are open sets 
$$
V_0\supseteq V_1\supseteq\dots\ni 1
$$
such that for any finite subsets $F_n, E_n\subseteq G$, $G\neq \bigcup_nF_nV_nE_n$.
\end{prop}

\begin{proof}Same proof as for Proposition \ref{Weil complete}, using the fact that any $2$-sided invariant compatible metric on $G$ is complete.
\end{proof}


\subsection{Narrow sequences and conjugacy classes}
While the results of the previous section essentially relied on various notions of completeness, we shall now investigate which role conjugacy classes play in $2$-sided coverings. For simplicity of notation, if $G$ is a Polish group and $g\in G$, we let $g^G=\{fgf\inv \del f\in G\}$ denote the conjugacy class of $g$.

\begin{thm}\label{small conj}
Suppose $G$ is a non-discrete Polish group such that the set 
$$
A=\{g\in G\del 1\in {\rm cl}(g^G)\}
$$
is not comeagre in any open neighbourhood of $1$.
Then there are open 
$$
V_0\supseteq V_1\supseteq \ldots\ni 1
$$
such that for any $f_n,g_n\in G$, $G\neq\bigcup_nf_nV_ng_n$.
\end{thm}

\begin{proof}Fix a compatible complete metric $d$ on $G$. Note that 
$$
A=\bigcap_n\{g\in G\del \e f\in G\;\; d(1,fgf\inv)<1/n\},
$$
which is a countable intersection of open sets. So, as $A$ is not comeagre in any neighbourhood of $1$, for any open $V\ni 1$ there must be some $n$ such that
$$
\{g\in G\del \e f\in G\;\; d(1,fgf\inv)<1/n\}
$$
is not dense in $V$ and therefore disjoint from some non-empty open $W\subseteq V$.
It follows that for any open $V\ni 1$ there  are non-empty open $W\subseteq V$ and  $U\ni 1$ such that $W\cap fUf\inv =\tom$ for any $f\in G$.

\begin{claim}
For any non-empty open $D\subseteq G$ there are $a,b\in D$ and an open set $U\ni 1$ such that for any $f,g\in G$, either 
$$
fUg\cap aU=\tom
$$
or 
$$
fUg\cap bU=\tom.
$$
\end{claim}
To see this, pick non-empty open sets $W\subseteq D\inv D$ and $U=U\inv\ni 1$ such that for any $f\in G$, $W\cap fUf\inv =\tom$. Let also $a,b\in D$ be such that $a\inv b\in W$. By shrinking $U$ if necessary, we can suppose that $Ua\inv bU\subseteq W$ and that 
$$
Ua\inv bU\cap\big( fUf\inv\cdot fUf\inv\big)=Ua\inv bU\cap fU^2f\inv=\tom
$$ 
for any $f\in G$.

Thus, for any $f,g\in G$, 
$$
(gaU)\inv \cdot gbU\cap (fUf\inv)\inv\cdot(fUf\inv)=\tom,
$$
whence either
$$
aU\cap g\inv fUf\inv=\tom
$$
or 
$$
bU\cap g\inv fUf\inv=\tom.
$$
Since $f,g$ are arbitrary, the claim follows.

Using the claim, we can inductively define open sets $V_0\supseteq V_1\supseteq\ldots\ni 1$ and $a_s\in G$ for $s\in 2^{<\N}\setminus \{\tom\}$ such that for any $s\in 2^n$, $i\in \{0,1\}$ and $f,g\in G$
\begin{itemize}
\item ${\rm diam}(a_{si}\overline V_n)<\frac1{n+1}$,
\item $a_{si}\overline V_n\subseteq a_sV_{n-1}$,
\item either $a_{s0}\overline V_n\cap fV_ng=\tom$ or $a_{s1}\overline V_n\cap fV_ng=\tom$.
\end{itemize}
So if $f_n, g_n\in G$ are given, we can inductively define $i_0,i_1,\ldots\in \{0,1\}$ such that 
$$
a_{i_0i_1\ldots i_n}\overline V_n\cap f_nV_ng_n=\tom
$$ 
for every $n$. It follows that $\bigcap_na_{i_0i_1\ldots i_n}\overline V_n$ is non-empty and disjoint from $\bigcup_nf_nV_ng_n$.
\end{proof}


\subsection{Narrow sequences in oligomorphic groups}
As mentioned earlier, the non-Archimedean Polish groups are exactly those isomorphic to closed subgroups of $S_\infty$. A closed subgroup $G\leqslant S_\infty$ is said to be {\em oligomorphic} if for any finite set $A\subseteq \N$ and $n\geqslant 1$, the pointwise stabiliser $G_A=\{g\in G\del \a x\in A\; g(x)=x\}$ induces only finitely many distinct orbits on $\N^n$. In particular, if one views  $\N$ as a discrete metric space with distance $1$ between all points, then an oligomorphic closed subgroup satisfies condition (3) of Proposition \ref{roelcke}, showing that it must be Roelcke precompact. In fact, Tsankov \cite{tsankov} characterised the Roelcke precompact closed subgroups of $S_\infty$ as those that can be written as inverse limits of oligomorphic groups.

\begin{thm}\label{oligo}
Let $G$ be an oligomorphic closed subgroup of $S_\infty$. Then there is a neighbourhood basis at $1$, $V_0\supseteq V_1\supseteq V_2\supseteq \ldots\ni 1$ such that 
$$
G\neq\bigcup_nF_nV_nE_n
$$
for all finite subsets $F_n,E_n\subseteq G$.
\end{thm}

\begin{proof}
Since, for any finite subset $A\subseteq \N$, the pointwise stabiliser $G_A$
induces only finitely many orbits on $\N$, it follows that the {\em algebraic closure} of $A$,
$$
\acl A=\{x\in \N\del G_A\cdot x\text{ is finite}\}
$$
is finite. Using that $G_{fA}\cdot fx=fG_Af\inv fx= f   G_A\cdot x$, one sees that $f\cdot \acl A=\acl {fA}$ for any $f\in G$, and, since $G_{\acl A}$ has finite index in $G_A$, we also find that $\acl{\acl A}=\acl A$. Finally, for any $n\geqslant 1$, $G$ induces only finitely many orbits on the space $\N^{[n]}$ of $n$-elements subsets of $\N$. So, as $|\acl A|=|f\cdot \acl A|=|\acl {fA}|$, the quantity
$$
M(n)=\max\big(|\acl A|\del A\in \N^{[n]}\big)
$$
is well-defined.

\begin{claim}
For all $g_i\in G$ and finite sets $B_i\subseteq \N$ with
$$
M(0)<|B_0|\leqslant |B_1|\leqslant\ldots
$$
and $|B_i|\til \infty$, we have $G\neq \bigcup_{i\in \N}g_iG_{B_i}$.
\end{claim}
To prove the claim, we will inductively construct finite algebraically closed sets  $A_i\subseteq \N$ and elements $f_i\in G$ such that the following conditions hold for all $i\in \N$,
\begin{enumerate}
\item $A_i\subseteq A_{i+1}$,
\item $f_{i+1}\in f_iG_{A_i}$,
\item $i\in A_i$,
\item $i\in f_iA_i$,
\item $f_iG_{A_i}\cap g_iG_{B_i}=\tom$,
\item for any $j\in \N$, if $B_j\subseteq A_i$, then also $f_iG_{A_i}\cap g_jG_{B_j}=\tom$.
\end{enumerate}

Assume first that this construction has been made. Then, as the $A_i$ is are increasing with $i$, we have by (2) that $f_i|_{A_i}= f_j|_{A_i}$ for all $i\leqslant j$. Also, by (3), for any $i\in \N$, $f_i(i)=f_{i+1}(i)=f_{i+2}(i)=\ldots$ and by  (4) $f_i\inv(i)\in A_i$, whence also  
$$
f_if_i\inv(i)=f_{i+1}f_i\inv(i)=f_{i+2}f_i\inv(i)=\ldots
$$
and so 
$$
f_i\inv(i)=f_{i+1}\inv(i)=f_{i+2}\inv(i)=\ldots.
$$
It follows that $(f_i)$ is both left and right Cauchy and thus converges in $G$ to some $f\in\bigcap_i f_iG_{A_i}$. Since by (5) we have $f_iG_{A_i}\cap g_iG_{B_i}=\tom$ it follows that $f\notin \bigcup_ig_iG_{B_i}$.

To begin the construction, set $A_{-1}=\acl\tom$ and $f_{-1}=1$. Then (6) hold since  $|A_{-1}|<|B_j|$ for all $j$.

Now, suppose $f_{-1},\ldots,f_i$ and $A_{-1},\ldots,A_i$ have been chosen such that (1)-(6) hold. Choose first $n\geqslant {i+1}$ large enough such that $M(|A_i|+3)<|B_{n+1}|$ and set
$$
C=\{i+1\}\cup B_0\cup\ldots \cup B_n\cup f_i\inv g_0B_0\cup\ldots\cup f_i\inv g_nB_n.
$$
As $A_i$ is algebraically closed, every orbit of $G_{A_i}$ on $\N\setminus A_i$ is infinite.  Therefore, by a lemma of P. M. Neumann \cite{neumann}, there is some $h\in G_{A_i}$ such that $h(C\setminus A_i)\cap (C\setminus A_i)=\tom$. Setting $f_{i+1}=f_ih\in f_iG_{A_i}$, we note that if $m\in C\setminus A_i$, then $h(m)\notin C$, whence
$$
f_{i+1}(m)=f_ih(m)\notin g_0B_0\cup\ldots\cup g_nB_n.
$$
If possible, choose $m\in B_{i+1}\setminus A_i$ and set
$$
A_{i+1}=\acl{A_i\cup\{m, i+1,f_{i+1}\inv(i+1)\}}.
$$ 
Otherwise, let 
$$
A_{i+1}=\acl{A_i\cup\{i+1,f_{i+1}\inv(i+1)\}}.
$$

That (1)-(4) hold are obvious by the choice of $A_{i+1}$. For (5), i.e., that 
$$
f_{i+1}G_{A_{i+1}}\cap g_{i+1}G_{B_{i+1}}=\tom,
$$ 
note that if $B_{i+1}\subseteq A_i$, then this is verified by condition (6) for the previous step and the fact that $f_{i+1}G_{A_{i+1}}\subseteq f_iG_{A_i}$. On the other hand, if  $B_{i+1}\not\subseteq A_i$, then there is some $m\in (A_{i+1}\cap B_{i+1})\setminus A_i$, whence $f_{i+1}(m)\notin g_{i+1}B_{i+1}$. In this case, $f_{i+1}(m)\neq g_{i+1}(m)$ and so 
$$
f_{i+1}G_{A_{i+1}}\cap g_{i+1}G_{B_{i+1}}=\tom.
$$
Finally, suppose that $B_j\subseteq A_{i+1}$ for some $j$, whence by the choice of $n$ we have $j\leqslant n$. If already $B_j\subseteq A_i$, then 
$$
f_{i+1}G_{A_{i+1}}\cap g_jG_{B_j}=\tom
$$
holds by (6) at the previous step. And if, on the other hand, there is some $m\in (A_{i+1}\cap B_j)\setminus A_i$, then $m\in C\setminus A_i$ and so $f_{i+1}(m)\notin g_jB_j$, whence
$$
f_{i+1}G_{A_{i+1}}\cap g_{j}G_{B_{j}}=\tom,
$$
which ends the construction and therefore verifies the claim.

To construct the neighbourhood basis $(V_n)$ at $1$, we pick finite algebraically closed sets $A_0\subseteq A_1\subseteq \ldots\subseteq \N=\bigcup_{n\in \N}A_n$ such that $M(0)<|A_0|$ and let $V_n=G_{A_n}$. Noting that for any $f,h\in G$, $fG_{A_n}h\inv= fh\inv G_{hA_n}$, we see that if $F_n$ and $E_n$ are finite subsets of $G$, we have
$$
\bigcup_nF_nV_nE_n=\bigcup_nF_nG_{A_n}E_n=\bigcup_ng_nG_{B_n}
$$
for some  $g_n\in G$ and $B_n\subseteq \N$ as in the claim. It thus follow that $G\neq \bigcup_nF_nV_nE_n$.
\end{proof}


\subsection{Narrow sequences in Roelcke precompact groups}
Suppose $G$ is a group acting by isometries on a metric space $(X,d)$. For any $n\geqslant 1$, we let $G$ act diagonally on $X^n$, i.e., 
$$
g\cdot (x_1,\ldots, x_n)=(gx_1,\ldots, gx_n),
$$
and equip $X^n$ with the supremum metric $d_\infty$ defined from $d$ by 
$$
d_\infty\big((x_1,\ldots,x_n),(y_1, \ldots,y_n)\big)=\sup_{1\leqslant i\leqslant n }d(x_i,y_i).
$$
Also, for any $\overline x\in X^n$ and $\eps>0$, let 
$$
V(\overline x, \eps)=\{g\in G\del d_\infty(g\overline x, \overline x)<\eps\}.
$$
We then have the following easy facts 
\begin{enumerate}
\item $V(\overline x,\eps)\cdot V(\overline x,\delta)\subseteq V(\overline x,\eps+\delta)$,
\item $f\cdot V(\overline x,\eps)\cdot f\inv =V(f\overline x,\eps)$,
\item for $|\overline x|=|\overline y|$, we have $V(\overline x,\sigma)\subseteq V\big(\overline y,\sigma+2d_\infty(\overline x,\overline y)\big)$,
\item for any subset $U\subseteq G$,
$$
d_H(U\cdot x,U\cdot y)\leqslant d(x,y),
$$
where $d_H$ denotes the Hausdorff distance induced by $d$.
\end{enumerate}
Also, an {\em $\eps$-ball} in $X$ is any subset of the form 
$$
B(x,\eps)=\{y\in X\del d(x,y)<\eps\}
$$
and if $D\subseteq X$ is any subset, we let 
$$
(D)_\eps=\{y\in X \del \e x\in D\; d(x,y)<\eps\}.
$$

For the following sequence of lemmas, suppose $G$ is a Roelcke precompact Polish group acting continuously and by isometries on a separable complete metric space $(X,d)$  with a dense orbit.

\begin{lemme}\label{covering}
For any $\eps>\delta>0$, $\sigma>0$ and $\overline x\in X^n$, the set
$$
\B(\overline x,\sigma, \delta)=\{y\in X\del V(\overline x,\sigma)\cdot y \text{ can be covered by finitely many $\delta$-balls}\}
$$
can be covered by finitely many $\eps$-balls.
\end{lemme}

\begin{proof}Let $\alpha>0$ be small enough such that $2\alpha<\sigma$ and $\delta+2\alpha<\eps$. Also, by Proposition \ref{roelcke} (3), let $A\subseteq X$ be a finite set such that $V(\overline x,\alpha)\cdot A$ is an $\alpha$-net in $X$.
Let also $C\subseteq A$ be the subset of all $z\in A$ such that $V(\overline x, \sigma-\alpha)\cdot z$ can be covered by finitely many $(\delta+\alpha)$-balls. Since $A$ and hence $C$ is finite, there are finitely many $(\delta+\alpha)$-balls $B_1,\ldots, B_k$ covering $V(\overline x,\sigma-\alpha)\cdot C$.

Now, suppose that $y\in X$ and $V(\overline x,\sigma)\cdot y$ can be covered by finitely many $\delta$-balls and find $z\in A$ and $g\in V(\overline x,\alpha)$ such that $d(gz,y)<\alpha$. Then
$$
V(\overline x,\sigma-\alpha)\cdot g\inv y\subseteq V(\overline x,\sigma-\alpha)\cdot V(\overline x,\alpha)\cdot y\subseteq V(\overline x,\sigma)\cdot y
$$
can be covered by finitely many $\delta$-balls and so, as
$$
d_H\big(V(\overline x,\sigma-\alpha)\cdot g\inv y,V(\overline x,\sigma-\alpha)\cdot z\big)\leqslant d(g\inv y,z)<\alpha,
$$
also $V(\overline x,\sigma-\alpha)\cdot z$ can be covered by finitely many $(\delta+\alpha)$-balls, i.e., $z\in C$. Since $\alpha<\sigma-\alpha$, we have $gz\in V(\overline x,\sigma-\alpha)\cdot z\subseteq B_1\cup\ldots\cup B_k$ and hence $y\in (B_1)_\alpha\cup\ldots\cup(B_k)_\alpha$. Since the $(B_i)_\alpha$ are each contained in the $\eps$-balls with the same centre, the result follows.
\end{proof}
We also note that
$$
\go B(g\overline x,\sigma,\delta)=g\cdot \go B(\overline x,\sigma,\delta).
$$ 
Using this, we can prove the following.

\begin{lemme} \label{uniform covering}
Suppose $\eps>\delta>0$, $\sigma>0$ and $n\geqslant 1$. Then there is some $k\geqslant 1$ such that for every $\overline z\in X^n$, $\go B(\overline z,\sigma,\delta)$ can be covered by $k$ many $\eps$-balls.
\end{lemme}

\begin{proof}
Choose $\alpha>0$ such that $2\alpha<\sigma$ and find by Proposition \ref{roelcke} (3) some $\overline x_1,\ldots,\overline x_p\in X^n$ such that 
$$
G\cdot \overline x_1\cup\ldots \cup G\cdot \overline x_p
$$
is $\alpha$-dense in $X^n$. By Lemma \ref{covering}, pick some $k\geqslant 1$ such that each $\go B(\overline x_i, \sigma-2\alpha,\delta)$ can be covered by $k$ many $\eps$-balls. 

Now suppose $\overline z\in X^n$ is given and find $g\in G$ and $1\leqslant i\leqslant p$ such that
$d_\infty(\overline x_i,g\overline z)<\alpha$. Then 
$$
V(\overline x_i,\sigma-2\alpha)\subseteq V(g\overline z,\sigma)
$$
and hence
\begin{displaymath}\begin{split}
g\cdot \go B(\overline z,\sigma,\delta)&=\go B(g\overline z,\sigma,\delta)\\
&=\{y\in X\del V(g\overline z,\sigma)\cdot y \text{ can be covered by finitely many $\delta$-balls}\}\\
&\subseteq \{y\in X\del V(\overline x_i,\sigma-2\alpha)\cdot y \text{ can be covered by finitely many $\delta$-balls}\}.
\end{split}\end{displaymath}
Since the latter can be covered by $k$ many $\eps$-balls, also $\go B(\overline z,\sigma,\delta)$ can be covered by $k$ many $\eps$-balls.
\end{proof}

The next statement is obvious.
\begin{lemme}
Suppose $\delta>0$, $\sigma>0$ and $\overline x\in X^n$. Then, for any $y\notin \go B(\overline x,\sigma,\delta)$ and any finite set $B_1,\ldots,B_k\subseteq X$ of $\delta$-balls, there is some $g\in V(\overline x,\sigma)$ such that $gy\notin B_1\cup\ldots\cup B_k$.
\end{lemme}

In the following, assume furthermore that $X$ is non-compact. Since $X$ is not totally bounded, we can find some $\eps>0$ such that $X$ contains an infinite $2\eps$-separated subset $D\subseteq X$.

\begin{lemme}\label{avoid}
For any $\sigma>0$, $0<\delta<\frac \eps2$ and $n\geqslant 1$, there is $\overline y \in X^{<\N}$ such that for any  $\overline x\in X^{n}$, finite $F\subseteq G$ and $h\in G$, there is $g\in V(\overline x, \sigma)$ such that 
$$
g\cdot V(\overline x\con h\inv \overline y, \delta)\cap F\cdot V(\overline y,\delta)\cdot h=\tom.
$$ 
\end{lemme}

\begin{proof}
By Lemma \ref{uniform covering}, we can find some $k\geqslant 1$ such that for any $\overline x\in X^n$ there are  $k$ many $\eps$-balls covering $\B(\overline x,\sigma,2\delta)$. Choose distinct $y_0,\ldots, y_k\in D$ and notice that for any $h\in G$, $\overline x\in X^n$ and $\eps$-balls   $B_1,\ldots,B_k$ covering $\B(\overline x,\sigma,2\delta)$, we have
$$
\{h\inv y_0,\ldots, h\inv y_k\}\not\subseteq B_1\cup\ldots\cup B_k,
$$ 
whence there is some $i=0,\ldots,k$ such that $h\inv y_i\notin \B(\overline x,\sigma,2\delta)$. We set $\overline y=(y_0,\ldots,y_k)$ and $V=V(\overline y,\delta)$.

Suppose now that $\overline x\in X^n$, $F\subseteq G$ is finite and $h\in G$. Then
$$
FVh=Fh\cdot V(h\inv \overline y,\delta)
$$
and we can pick some $i=0,\ldots,k$ such that $h\inv y_i\notin \B(\overline x,\sigma,2\delta)$. It follows that there is some $g\in V(\overline x,\sigma)$ such that 
$$
d(gh\inv y_i,fy_i)\geqslant 2\delta
$$
for all $f\in F$, whence 
$$
g\cdot V(h\inv y_i,\delta)\cap Fh\cdot V(h\inv y_i,\delta)=\tom
$$
and thus also
$$
g\cdot V(\overline x\con h\inv \overline y, \delta)\cap FVh=\tom,
$$ 
which proves the lemma.
\end{proof}

\begin{thm}\label{approx oligo}
Suppose $G$ is a non-compact, Roelcke precompact Polish group. Then there is a neighbourhood basis at $1$, $V_0\supseteq V_1\supseteq \ldots \ni 1$ such that for any $h_n\in G$ and finite $F_n\subseteq G$,
$$
G\neq \bigcup_n F_nV_nh_n.
$$
\end{thm}

\begin{proof}
By Proposition \ref{roelcke} (4), without loss of generality we can suppose that $G$ is a closed subgroup of ${\rm Isom}(X,d)$, where $(X,d)$ is a separable complete metric space with a dense $G$-orbit. Since $G$ in non-compact, so is $X$ and thus there is an $\eps>0$ such that $X$ contains an infinite $2\eps$-separated subset $D\subseteq X$.

Let $z_1,z_2,\ldots$ be a dense sequence in $X$  with each point listed infinitely often and set $\delta_n=\frac \eps{3^n}$. So $\sum_{n\geqslant m}\delta_n\Lim{m\til \infty}0$.

We define inductively tuples $\overline y_i\in X^{<\N}$ and numbers $n_i$ as follows. First, let $n_1=1$ and apply  Lemma \ref{avoid} to $\sigma=\delta=\delta_1$ and $n=n_1$ to get $\overline y_1$. In general, if $n_i$ and $\overline y_i$ are chosen, we set $n_{i+1}=n_i+|\overline y_i|+2$ and apply Lemma \ref{avoid} to $\sigma=\delta=\delta_{i+1}$ and $n=n_{i+1}$ to find $\overline y_{i+1}$. Finally, set $V_i=V(\overline y_i,\delta_i)$.

Thus, for any $\overline x\in X^{n_i}$, finite $F\subseteq G$ and $h\in G$, there is some $g\in V(\overline x,\delta_i)$ such that
$$
g\cdot V(\overline x\con h\inv \overline y_i,\delta_i)\cap FV_ih=\tom.
$$

Now suppose $F_i\subseteq G$ are finite subsets and $h_i\in G$. Set $\overline x_1=z_1$. By induction on $i$, we now construct $\overline x_i\in X^{<\N}$ and $g_i\in G$ such that
\begin{enumerate}
\item $\overline x_i\in X^{n_i}$,
\item $\overline x_{i+1}=\big(\overline x_i,h_i\inv \overline y_i, z_i, (g_1\cdots g_i)\inv(z_i)\big)$,
\item $g_i\in V(\overline x_i,\delta_i)$,
\item $g_1g_2\cdots g_i\cdot V(\overline x_i\con h_i\inv \overline y_i,\delta_i)\cap F_iV_ih_i=\tom$.
\end{enumerate}

Note that since $\delta_{i+1}+\delta_{i+1}<\delta_{i}$, we have by  (2) and (3) above
\begin{displaymath}\begin{split}
g_1g_2\cdots g_{i}g_{i+1}\cdot \overline{V(\overline x_{i+1},\delta_{i+1})}
&\subseteq g_1g_2\cdots g_{i}\cdot V(\overline x_{i+1},\delta_{i+1})\overline{V(\overline x_{i+1},\delta_{i+1})}\\
&\subseteq g_1g_2\cdots g_{i}\cdot V(\overline x_{i+1},\delta_{i})\\
&\subseteq g_1g_2\cdots g_{i}\cdot V(\overline x_{i}\con h_i\inv \overline y_i,\delta_{i})\\
&\subseteq g_1g_2\cdots g_{i}\cdot \overline{V(\overline x_{i},\delta_{i})},
\end{split}\end{displaymath}
and so, in particular,
$$
\bigcap_{i=1}^\infty g_1\cdots g_i\cdot V(\overline x_i\con h\inv_i\overline y_i,\delta_i)= \bigcap_{i=1}^\infty g_1\cdots g_i\cdot \overline{V(\overline x_i,\delta_i)},
$$
which is disjoint from $\bigcup_{i=1}^\infty F_iV_ih_i$.

We claim that $g_1g_2g_3\cdots$ converges pointwise  to a surjective isometry from $X$ to $X$ and hence converges in $G$. 

First, to see that it converges pointwise on $X$ to an isometry $\psi$ from $X$ to $X$, it suffices to show that  $g_1g_2g_3\cdots(z_l)$ converges for any $l$, i.e., that 
$\big(g_1g_2\cdots g_i(z_l)\big)_{i=1}^\infty$ is Cauchy in $X$.
But
\begin{displaymath}\begin{split}
d\big(g_1\cdots g_i(z_l), g_1\cdots g_{i+m}(z_l)\big)
=&d\big(z_l, g_{i+1}\cdots g_{i+m}(z_l)\big)\\
\leqslant& d\big(z_l, g_{i+1}(z_l)\big)+d\big( g_{i+1}(z_l),g_{i+1}g_{i+2}(z_l)\big)\\
&+\ldots+d\big(g_{i+1}\cdots g_{i+m-1}(z_l), g_{i+1}\cdots g_{i+m}(z_l)\big)\\
\leqslant& d\big(z_l, g_{i+1}(z_l)\big)+d\big(z_l, g_{i+2}(z_l)\big)+\ldots+d\big(z_l, g_{i+m}(z_l)\big)\\
<& \delta_{i+1}+\delta_{i+2}+\ldots+\delta_{i+m}\\
<&\delta_i,
\end{split}\end{displaymath}
whenever $i\geqslant l$, showing that the sequence is Cauchy.

To see the pointwise limit $\psi$ is surjective, it suffices to show that the image of $\psi$ is dense in $X$. So fix $z_l$ and $\eps>0$. We show that ${\rm Im}(\psi)\cap B(z_l,\eps)\neq \tom$. Begin by choosing $i$ large enough such that $\delta_i<\eps$ and $z_l=z_i$. Then $z_l=z_i=g_1\cdots g_i(x)$ for some term $x$ in $\overline x_{i+1}$ and hence, by a calculation as above, we find that
\begin{displaymath}\begin{split}
d(z_l, g_1\cdots g_{i+m}(x))=d\big(g_1\cdots g_i(x), g_1\cdots g_{i+m}(x)\big)<\delta_i
\end{split}\end{displaymath}
for any $m\geqslant 0$, whence $\psi(x)=\lim_{m\til \infty}g_1\cdots g_{i+m}(x)$ is within $\eps$ of $z_l$.

Finally, since the sets $g_1\cdots g_i\cdot \overline{V(\overline x_i,\delta_i)}$ are decreasing and 
$$
g_1\cdots g_i\in g_1\cdots g_i\cdot \overline{V(\overline x_i,\delta_i)}
$$
for every $i$, we have 
$$
\psi=\lim_ig_1\cdots g_i\in \bigcap_i g_1\cdots g_i\cdot \overline{V(\overline x_i,\delta_i)}
$$
and hence $\psi\in G\setminus \bigcup_iF_iV_ih_i$, which finishes the proof.
\end{proof}

In \cite{malicki}, Malicki studied Solecki's question of whether any feebly locally compact Polish group is locally compact and proved among other things that none of ${\rm Isom}(\U_1)$, $\ku U(\ell_2)$ nor oligomorphic closed subgroups of $S_\infty$ are feebly non-locally compact. Theorem \ref{oligo} strengthens his result for oligomorphic closed subgroups of $S_\infty$, but does not imply his results for  ${\rm Isom}(\U_1)$ and $\ku U(\ell_2)$.  We do not know if Theorem \ref{approx oligo} can be strengthened to two-sided translates $F_nV_nE_n$, where $F_n,E_n\subseteq G$ are arbitrary finite subsets.


\subsection{Isometric actions defined from narrow sequences}\label{isometric actions}
Using the narrow sequences $(V_n)$ defined hitherto, we shall now proceed to construct (affine) isometric actions of Polish groups with interesting dynamics. The basic underlying construction for this was previously used by Nguyen Van Th\'e and Pestov \cite{pestov} in their proof of the equivalence of (3) and (4) of Theorem \ref{compact}.
\begin{defi}
Suppose $G$ is a Polish group acting continuously and by isometries on a separable complete metric space $(X,d)$. We say that 
\begin{enumerate}
\item $G$ is {\em strongly point moving} if there are $\eps_n>0$ such that for all $x_n,y_n\in X$ there is some $g\in G$ satisfying $d(gx_n,y_n)>\eps_n$ for all $n\in\N$,
\item $G$ is {\em strongly compacta moving} if there are $\eps_n>0$ such that for all compact subsets $C_n\subseteq X$ there is some $g\in G$ satisfying ${\rm dist}(gC_n,C_n)>\eps_n$ for all $n\in \N$.
\end{enumerate}
\end{defi}
Here the distance between two compact sets is the minimum distance between points in the two sets.

We can now reformulate the existence of narrow sequences in Polish groups by the existence of strongly point moving isometric actions as follows.
\begin{thm}\label{moving points}
Let $G$ be a Polish group. Then the following are equivalent.
\begin{enumerate}
\item $G$ admits a strongly point moving continuous isometric action on a separable complete metric space,  
\item $G$ has a neighbourhood basis $(V_n)$ at $1$ such that for any $f_n, h_n\in G$, $G\neq \bigcup_nf_nV_nh_n$.
\end{enumerate}
\end{thm}

\begin{proof}(2)$\saa$(1): Suppose (2) holds and let $V_0\supseteq V_1\supseteq \ldots \ni 1$ be the given neighbourhood basis at $1$. Fix also a compatible left-invariant metric $d$ on $G$. Then, by decreasing each $V_n$, we can suppose that 
$$
V_n=\{g\in G\del d(g,1)<3\eps_n\}
$$
for some $\eps_n>0$. Note then that, for $f,g,h\in G$, we have
$$
g\notin fV_nh
\equi 
f\inv gh\inv \notin V_n
\equi 
d(f\inv gh\inv ,1)\geqslant 3\eps_n
\equi
d(gh\inv,f)\geqslant 3\eps_n.
$$
Let now $(X,d)$ denote the completion of $G$ with respect to $d$ and consider the extension of the left shift action of $G$ on itself to $X$. Then, for any $x_n, y_n\in X$, there are $f_n,h_n\in G$ such that $d(x_n,h\inv_n)<\eps_n$ and $d(y_n,f_n)<\eps_n$, whence, if $g\notin \bigcup_{n}f_nV_nh_n$, we have $d(gh_n\inv,f_n)\geqslant 3\eps_n$ and thus $d(gx_n,y_n)>\eps_n$ for all $n$. Thus, $G$ is strongly point moving.

(1)$\saa$(2): Let $\eps_n>0$ be the constants given by a strongly point moving isometric action of $G$ on a separable complete metric space $X$ and let $x\in X$ be fixed. Set $W_n=\{g\in G\del d(gx,x)<\eps_n\}$ and let $V_n\subseteq W_n$ be open subsets such that $(V_n)$ forms a neighbourhood basis at $1$. Then, if $f_n, h_n\in G$ are given, find some $g$ such that $d(gh_n\inv x, f_nx)>\eps_n$ for all $n$. It thus follows that $g\notin \bigcup_nf_nV_nh_n$.
\end{proof}

The following lemma is proved in a slightly different but equivalent setup in \cite{pestov}.
\begin{lemme}\label{pestov}
Suppose $G$ is a group acting by isometries on a metric space $(X,d)$, $e\in X$ and let $\rho_e$ denote the affine isometric action of $G$ on $\AE(X)$ defined by 
$$
\rho_e(g)m=m(g\inv\;\cdot\;)+m_{ge,e}.
$$
Then for any $m_1,m_2\in \AE(X)$ there are finite sets $F,E\subseteq X$ such that if  $g\in G$ satisfies $d(gF,E)>2\eps$, then also  $\norm{\rho_e(g)m_1-m_2}>\eps$.
\end{lemme}

\begin{proof}
By approximating each $m_i$ by a molecule within distance $\frac \eps2$, it suffices to show that for any molecules $m_1,m_2\in \M(X)$ there are finite sets $F,E\subseteq G$ such that if  $d(gF,E)>2\eps$, then also $\norm{\rho_e(g)m_1-m_2}\geqslant2\eps$. For this, write $m_1=\sum_{i=1}^na_im_{x_i,y_i}$, $m_2=\sum_{i=1}^kb_im_{v_i,w_i}$ and and assume that  $g\in G$ satisfies
$$
d(\{gx_i,gy_i,ge\}_{i=1}^n, \{v_i,w_i,e\}_{i=1}^k)>2\eps.
$$
Now, letting $f\colon X\til \R$ be defined by 
$$
f(x)= \min\big\{2\eps, d(x, \{v_i,w_i,e\}_{i=1}^k)\big\}
$$
we see that $f$ is $1$-Lipschitz and $f(x)=2\eps$ for all $x\in \{gx_i,gy_i,ge\}_{i=1}^n$, while $f(x)=0$ for $x \in \{v_i,w_i,e\}_{i=1}^k$.
It follows that
\begin{displaymath}\begin{split}
\norm{\rho_e(g)m_1-m_2}
&=\Norm{\sum_{i=1}^na_im_{gx_i,gy_i}+m_{ge,e}-\sum_{i=1}^kb_im_{v_i,w_i}  } \\
&=\sum_{i=1}^na_i\big(f(x_i)-f(y_i)\big)+\big(f(ge)-f(e)\big)+\sum_{i=1}^kb_i\big(f(v_i)-f(w_i)\big)\\
&=2\eps,
\end{split}\end{displaymath}
which proves the lemma.
\end{proof}

\begin{thm}\label{moving compacta}
Let $G$ be a Polish group. Then the following are equivalent.
\begin{enumerate}
\item $G$ admits a strongly compacta moving continuous affine isometric action on a separable Banach space,  
\item $G$ has a neighbourhood basis $(V_n)$ at $1$ such that for any finite subsets $F_n,E_n\subseteq  G$, $G\neq \bigcup_nF_nV_nE_n$.
\end{enumerate}
\end{thm}

\begin{proof}
The implication (1)$\saa$(2) is similar to that of Theorem \ref{moving points}. Also for (2)$\saa$(1), let again $d$ be a compatible left-invariant metric on $G$ and let $G$ act on itself on the left. As before, we can suppose that
$V_n=\{g\in G\del d(g,1)<6\eps_n\}$ for some $\eps_n>0$  and thus  for all finite sets $F_n, E_n\subseteq G$ there is some $g\in G$ such that  $d(gF_n,E_n)\geqslant6\eps_n$ for all $n\in \N$. 
Fix an arbitrary element $e\in G$ and let $\rho_e$ denote the affine isometric action of $G$ on $\AE(G)$ defined by $\rho_e(g)m=m(g\inv\;\cdot\;)+m_{ge,e}$. 

Assume now that $C_n\subseteq \AE(G)$ are compact and find finite $\eps_n$-dense subsets $M_n\subseteq C_n$. By Lemma \ref{pestov}, there are finite subsets $F_n,E_n\subseteq G$ such that if $d(gF_n,E_n)>6\eps_n$, then $\norm{\rho_e(g)m_1-m_2}>3\eps_n$ for all $m_1, m_2\in M_n$, whereby also $\norm{\rho_e(g)m_1-m_2}>\eps_n$ for all $m_1, m_2\in C_n$, finishing the proof of the theorem.
\end{proof}

Combining Theorems \ref{moving points}, \ref{moving compacta}, \ref{oligo}, \ref{approx oligo}, \ref{small conj} and Propositions \ref{loc comp} (b), \ref{two-sided invariant}, we obtain the following two corollaries.

\begin{cor}The following classes of Polish groups admit strongly point moving, continuous isometric actions on  separable complete metric spaces,
\begin{itemize}
\item non-discrete, unimodular, locally compact groups,
\item non-compact, Roelcke precompact groups,
\item Polish groups $G$ such that 
$\{g\in G\del 1\in {\rm cl}(g^G)\}$
is not comeagre in any neighbourhood of $1$.
\end{itemize}
\end{cor}

\begin{cor}The following classes of Polish groups admit strongly compacta moving, continuous affine isometric actions on  separable Banach spaces,
\begin{itemize}
\item non-locally compact SIN groups,
\item oligomorphic closed subgroups of $S_\infty$.
\end{itemize}
\end{cor}

\end{document}